\documentclass[12pt]{amsart}
\pdfoutput=1

\usepackage{amsmath}
\usepackage{caption}
\usepackage{amssymb}
\usepackage{amscd}
\usepackage{latexsym}
\usepackage{float,graphicx}
\usepackage{mathtools}
\usepackage{array}
\usepackage{enumitem}
\usepackage{mathrsfs}
\usepackage{xfrac}
\usepackage{bm}
\usepackage{bbm}
\usepackage{nccmath}
\usepackage{hyperref}
\usepackage{multirow}
\usepackage{tikz}
\usepackage{tikz-cd}
\usepackage[T1]{fontenc}
\usepackage[utf8]{inputenc}
\usepackage{etex}
\usepackage[all]{xy}
\usepackage[protrusion=true,expansion=true]{microtype}
\usepackage[percent]{overpic}
\usepackage[margin=42mm]{geometry}
\usepackage{lipsum}

\tikzset{>=latex,auto}
\usetikzlibrary{positioning,shapes.geometric}

\newtheorem{theorem}{Theorem}[section]
\newtheorem{lemma}[theorem]{Lemma}
\newtheorem{corollary}[theorem]{Corollary}

\theoremstyle{definition}
\newtheorem{definition}[theorem]{Definition}
\newtheorem{example}[theorem]{Example}

\theoremstyle{remark}
\newtheorem{remark}[theorem]{Remark}
\newtheorem{question}[theorem]{Question}

\numberwithin{equation}{section}

\newcommand{\thmref}[1]{Theorem~\ref{#1}}

\newcommand{\secref}[1]{\S\ref{#1}}
\newcommand{\lemref}[1]{Lemma~\ref{#1}}
\newcommand{\defref}[1]{Definition~\ref{#1}}
\newcommand{\corref}[1]{Corollary~\ref{#1}}
\newcommand{\figref}[1]{Fig.~\ref{#1}}
\newcommand{\exref}[1]{Example~\ref{#1}}

\newcommand{\vs}{\vspace{6pt}}
\newcommand{\CC}{{\mathbb C}}

\newcommand{\RR}{{\mathbb R}}

\newcommand{\TT}{{\mathbb T}}
\newcommand{\ZZ}{{\mathbb Z}}

\newcommand{\OO}{{\mathcal O}}
\newcommand{\sS}{{\mathcal S}}
\newcommand{\sR}{{\mathcal R}}
\newcommand{\sC}{{\mathcal C}}
\newcommand{\sG}{{\mathcal G}}

\newcommand{\id}{\operatorname{id}}
\newcommand{\md}{\mathbf{m}_d}
\newcommand{\mk}{\mathbf{m}_k}

\newcommand{\des}{\operatorname{des}}
\newcommand{\sig}{\mathbf{sig}}
\newcommand{\dep}{\mathbf{dep}}
\newcommand{\fix}{\mathbf{fix}}
\newcommand{\sym}{\operatorname{sym}}
\newcommand{\ds}{\displaystyle}

\newcommand{\sm}{\smallsetminus}

\newcommand{\bv}{\boldsymbol{v}}

\newcommand{\bn}{\boldsymbol{n}}
\newcommand{\bl}{\boldsymbol{\ell}}
\newcommand{\bx}{\boldsymbol{x}}
\newcommand{\by}{\boldsymbol{y}}
\newcommand{\bp}{\boldsymbol{p}}
\newcommand{\bw}{\boldsymbol{w}}

\newcommand{\Mbrot}{\mathcal M}

\newcommand{\bit}{\it \bfseries}

\begin{document}

\title[Combinatorial types of periodic orbits]{On combinatorial types of periodic orbits of the map $\boldsymbol{x \mapsto kx}$ (\MakeLowercase{mod} $\boldsymbol{\ZZ}$)} 

\author[C. L. Petersen and S. Zakeri]{Carsten L. Petersen and Saeed Zakeri}

\address{Department of Mathematics, Roskilde University, DK-4000 Roskilde, Denmark} 

\email{lunde@ruc.dk}

\address{Department of Mathematics, Queens College of CUNY, 65-30 Kissena Blvd., Queens, New York 11367, USA} 

\address{Department of Mathematics, The Graduate Center of CUNY, 365 Fifth Ave., New York, NY 10016, USA}

\email{saeed.zakeri@qc.cuny.edu}

\date{August 6, 2019}

\begin{abstract}
We study the combinatorial types of periodic orbits of the standard covering endomorphisms $\mk(x)=k x \ (\text{mod} \ \ZZ)$ of the circle for integers $k \geq 2$ and the frequency with which they occur. For any $q$-cycle $\sigma$ in the permutation group $\sS_q$, we give a full description of the set of period $q$ orbits of $\mk$ that realize $\sigma$ and in particular count how many such orbits there are. The description is based on an invariant called the ``fixed point distribution'' vector and is achieved by reducing the realization problem to finding the stationary state of an associated Markov chain. Our results generalize earlier work on the special case where $\sigma$ is a rotation cycle, and can be viewed as a missing combinatorial ingredient for a proper understanding of the dynamics of complex polynomial maps of degree $\geq 3$ and the structure of their parameter spaces.
\end{abstract}

\maketitle

\tableofcontents

\section{Introduction}\label{intro}

Let $k$ be an integer $\geq 2$ and $\mk: \RR/\ZZ \to \RR/\ZZ$ be the multiplication-by-$k$ map of the circle defined by $\mk(x)=k x \ (\text{mod} \ \ZZ)$. This paper studies the combinatorial types of the periodic orbits of $\mk$ and the frequency with which they occur. Here two periodic orbits $\OO, \OO'$ of $\mk$ have the same combinatorial type if there is an orientation-preserving homeomorphism $\phi: \RR/\ZZ \to \RR/\ZZ$ which maps $\OO$ to $\OO'$ and satisfies $\phi \circ \mk = \mk \circ \phi$ on $\OO$. \vs 

We can formulate the problem more formally as follows. Let $\sS_q$ denote the group of permutations of $q$ objects, identified with bijections of $\ZZ/q\ZZ$, and let $\sC_q \subset \sS_q$ be the collection of $q$-cycles, namely permutations that act transitively on $\ZZ/q\ZZ$. We define the {\it combinatorial type} of $\sigma \in \sC_q$ as its conjugacy class $[\sigma]$ under the action of the rotation group $\sR_q$ generated by $\rho: i \mapsto i+1 \ (\text{mod} \ q)$. A period $q$ orbit $\OO= \{ x_1, \ldots, x_q \}$ of a continuous map $f: \RR/\ZZ \to \RR/\ZZ$, labeled so that $0 \leq x_1<\cdots<x_q<1$, is said to {\it realize} $\sigma \in \sC_q$ if $f(x_i)=x_{\sigma(i)}$ for all $i$. We say that $\OO$ realizes the combinatorial type $[\sigma]$ if it realizes $\rho^j \sigma \rho^{-j} \in \sC_q$ for some $j$. \vs 

Whether a given $\sigma \in \sC_q$ has a chance to be realized by a period $q$ orbit of $\mk$ depends on a measure of complexity of $\sigma$ known as its {\it descent number} $\des(\sigma)$ (see \defref{desn}). This is an integer between $1$ and $q-1$. In \cite{Mc}, McMullen gave a topological interpretation for $\des(\sigma)$ as the minimum degree of a covering map of the circle with a period $q$ orbit that realizes $\sigma$. This gives the necessary upper bound $\des(\sigma) \leq k$ for any $\sigma$ realized by $\mk$. The present paper will prove that this condition is also sufficient; in fact, we develop a machinery to describe the set of all period $q$ orbits of $\mk$ that realize a given $\sigma$. This, in particular, allows us to count precisely how many of the $O(k^q)$ period $q$ orbits of $\mk$ realize $\sigma$, and how many realize the combinatorial type $[\sigma]$. The main idea is to reduce the realization problem to finding the stationary state of a regular Markov chain. \vs          

The case $\des(\sigma)=1$ is of special interest in complex dynamics. It is easy to see that $\des(\sigma)=1$ if and only if $\sigma=\rho^p$ for some $p$ relatively prime to $q$. The periodic orbits of $\mk$ that realize such $\sigma$ are {\it rotation sets} with combinatorial rotation number $p/q$. This case was originally studied in Goldberg \cite{G} and further utilized by Goldberg and Milnor \cite{GM} in their work on fixed point portraits of complex polynomial maps. See also \cite{Mc} and \cite{Z1} for alternative approaches to the realization problem in this case. The question of which cycles are realizable under the doubling map $\mathbf{m}_2$ is also addressed in \cite{BS}. \vs 

Here is a brief outline of the paper and main results. Sections \S \ref{CTC}-\S \ref{realcyc} provide the preliminary material: \S \ref{CTC} gives the algebraic background on $q$-cycles in the permutation group $\sS_q$ and their combinatorial types, as well as the descent number of a permutation and its basic properties. \S \ref{sec:CMFP} is a quick review of the well-known properties of covering endomorphisms of the circle and their fixed points. \S \ref{realcyc} discusses realizations of a given $\sigma \in \sC_q$ under general covering maps of the circle and shows that the combinatorial invariant $\des(\sigma)$ is the minimal topological degree of such covering maps. \vs

In \S \ref{TMC} we assign a {\it transition matrix} $A$ to each $\sigma \in \sC_q$; this is a $q \times q$ binary matrix that encodes the dynamics of the partition intervals $I_i = [x_i, x_{i+1}]$ under any minimal realization of $\sigma$. The main diagonal entries of $A$ define the {\it signature} $\sig(\sigma)=(a_1, \ldots, a_q)$, where $a_i=1$ or $0$ according as $I_i$ maps over itself or not. The number of the entries of $1$ in $\sig(\sigma)$ is $d-1$, where $d=\des(\sigma)$. The rescaled matrix $(1/d)A$ is stochastic and regular when $d>1$, so by classical Perron-Frobenius theory it has a unique probability eigenvector corresponding to the leading eigenvalue $\lambda=1$. \vs 

In \S \ref{sec:min} we use the existence of this eigenvector to prove that a given $\sigma \in \sC_q$ with $\des(\sigma)=d>1$ and $\sig(\sigma)=(a_1, \ldots, a_q)$ is realized by a unique orbit $\OO$ of $\md$ if and only if $a_q=1$ (\thmref{main1}). In this case the rotated copies $\OO-j/(d-1) \ (\text{mod} \ \ZZ)$ will form all realizations of the combinatorial type $[\sigma]$ under $\md$. The number of such realizations is $(d-1)/s$, where $s=\sym(\sigma)$ is the {\it symmetry order} of $\sigma$, i.e., the order of the stabilizer group of $\sigma$ in $\sR_q$ (Theorems \ref{main2} and \ref{rotcop}). \vs        

\S \ref{GEN} generalizes the realization theorem to all higher degrees. Given $\sigma \in \sC_q$ with $\des(\sigma)=d \geq 1$ and an integer $k>d$, we reduce the realization problem of $\sigma$ under $\mk$ to an eigenvector problem for an associated stochastic matrix $(1/k)(A+P)$, where $A$ is the transition matrix of $\sigma$ and $P$ is a $q \times q$ binary matrix characterized by the choice of how many fixed points of $\mk$ would fall between any consecutive pair of points in the realizing orbit. The structure of $A$ (more precisely, the position of $1$'s in the signature $\sig(\sigma)$) prescribes the relative location of $d-1$ of the total $k-1$ fixed points of $\mk$; the number of remaining ``free'' choices is $k-d$ if $a_q=1$ and $k-d-1$ if $a_q=0$. We prove that all admissible choices of $P$ will lead to realizing orbits under $\mk$ (\thmref{main3}). This shows that the number of realizations of $\sigma$ under $\mk$ is the binomial coefficient ${q+k-d \choose q}$ if $a_q=1$ and is ${q+k-d-1 \choose q}$ if $a_q=0$ (\thmref{main4}). The combinatorial type $[\sigma]$ will then have $\frac{k-1}{s} \ {q+k-d-1 \choose q-1}$ realizations under $\mk$, where again $s$ is the symmetry order of $\sigma$ (\thmref{main5}). In the special case $d=1$, our theorem recovers the count ${q+k-2 \choose q}$ in \cite{G} for the number of distinct rotation sets of a given rotation number $p/q$ under $\mk$. \vs 

The above realizations, corresponding to admissible choices of $P$, are parametrized by their {\it fixed point distribution} $\fix(\OO)=(n_1, \ldots, n_q)$, where $n_i$ is the number of fixed points of $\mk$ in the partition interval $I_i$ defined by the orbit $\OO$. Alternatively, they can be parametrized by a dual invariant called the {\it (cumulative) deployment vector} $\dep(\OO)=(w_1,\ldots, w_{k-1})$, where $w_i$ is the number of elements of $\OO$ that belong to $(0,i/(k-1))$ (see \cite{G} and compare \cite{Z1}). If $\sig(\sigma)=(a_1, \ldots, a_q)$ and if $i_1<\cdots<i_{d-1}$ are the indices for which $a_i=1$, then $\dep(\OO)=(w_1, \ldots, w_{k-1})$ satisfies the following conditions: \vs
\begin{enumerate}
\item[(i)]
$0 \leq w_1 \leq \cdots \leq w_{k-1}=q$, \vs
\item[(ii)]
$i_1, \ldots, i_{d-1}$ appear among the components $w_1, \ldots, w_{k-1}$, \vs
\end{enumerate} 
where (ii) holds vacuously if $d=1$. Any integer vector $\bw = (w_1, \ldots, w_{k-1})$ which satisfies (i) and (ii) is called {\it $\sigma$-admissible}. Our Theorems \ref{main1} and \ref{main3} can then be unified and restated as follows: \vs 

\noindent
{\bf The Realization Theorem.} \ {\it Let $\sigma \in \sC_q$ with $\des(\sigma)=d$, and take an integer $k \geq \max \{ d, 2 \}$. Then, for every $\sigma$-admissible vector $\bw \in \ZZ^{k-1}$ there is a unique realization $\OO$ of $\sigma$ under $\mk$ with $\dep(\OO)=\bw$.} \vs \vs

Similarly, our Theorems \ref{main1}, \ref{main2}, \ref{main4} and \ref{main5} can be combined into the following unified statement which covers all the cases involved: \vs 

\noindent
{\bf The Counting Theorem.} \ {\it Let $\sigma \in \sC_q$ with $\des(\sigma)=d$, $\sym(\sigma)=s$ and $\sig(\sigma)=(a_1,\ldots, a_q)$. Then, for every integer $k \geq \max \{ d, 2 \}$ the number of realizations of $\sigma$ under $\mk$ is
$$
{q+k-d+a_q-1 \choose q},
$$
while the number of realizations of the combinatorial type $[\sigma]$ under $\mk$ is
$$
\frac{k-1}{s} \, {q+k-d-1 \choose q-1}. \vs
$$
}

The abstract problem studied here is motivated by its connection to complex dynamics. The link is provided by the fact that for each $k \geq 2$ the periodic points of $\mk$ appear as the angles of the dynamic rays that land on the repelling or parabolic periodic points of complex polynomial maps of degree $k$. Two periodic points of $\mk$ with distjoint orbits are {\it compatible} if there is a polynomial whose corresponding dynamic rays {\it co-land} at the same periodic point. The co-landing patterns of periodic dynamic rays lead to a partition of the space of all polynomials of degree $k$ into strata that are more amenable to investigation. This idea has proved to be instrumental in understanding the dynamics of individual polynomials and the structure of their parameter spaces. In the work of Thurston \cite{T} the angles of co-landing periodic rays give rise to the {\it lamination} associated with a polynomial. A similar theme is central to the work of Douady and Hubbard \cite{DH} on the quadratic family $\{ Q_c: z \mapsto z^2+c \}_{c \in \CC}$. They show that every parabolic parameter $c_0 \neq 1/4$ in the Mandelbrot set $\Mbrot$ is the landing point of precisely two parameter rays with periodic angles $\theta, \theta'$ of the same period $q \geq 2$ under $\mathbf{m}_2$. The open ``wake'' $W(\theta,\theta')$ cut out by the union of these rays and their common landing point $c_0$ can be characterized as the set of all $c$ for which the dynamic rays at angles $\theta, \theta'$ co-land at the ``$\beta$-fixed point'' of the renormalization (i.e., a ~quadratic-like restriction of the iterate $Q_c^{\circ q}$ around the critical value $c$) and the connectedness locus of the family of these renormalizations is a homeomorphic copy of $\Mbrot$ contained in $W(\theta,\theta')\cup\{c_0\}$. This effectively labels all renormalization copies of $\Mbrot$ inside $\Mbrot$. The stability of co-landing of the dynamic rays with angles $\theta, \theta'$ throughout $W(\theta,\theta')$ was subsequently used by Yoccoz to define a Markov partition of the filled Julia set of $Q_c$ known as the {\it Yoccoz puzzle} which ultimately led to his ground-breaking work on local connectivity of $\Mbrot$ at all finitely renormalizable parameters. \vs

A full combinatorial description of the connectedness loci in higher degrees has been partially attempted but needs deeper exploration. A common idea is to look at the patterns in degree $k+1$ that are inherited from degree $k$. For example, by comparing the internal and external angles in the dynamical plane of suitable cubic polynomials, one can find a natural way to associate to every period $q$ point of $\mathbf{m}_2$ a pair of ``neighboring'' compatible orbits of $\mathbf{m}_3$ that realize the same cycle in $\sC_q$. More precisely, for each period $q$ orbit $\{ z_1, \ldots, z_q \}$ of $\mathbf{m}_2$ with $0<z_1<\cdots<z_q<1$ and each $1 \leq j \leq q$ there are ``neighboring'' orbits $\{ x_1, \ldots, x_q \}, \{ y_1, \ldots, y_q \}$ of $\mathbf{m}_3$ such that \vs     
\begin{enumerate}
\item[$\bullet$]
They are interlaced:
$$
0<x_1<y_1< \cdots < x_q< y_q < 1;
$$

\item[$\bullet$]
$y_j-x_j = \dfrac{3^{q-1}}{3^q-1} > \dfrac{1}{3}$, so $\mathbf{m}_3$ maps $[x_j,y_j]$ over the whole circle;
 
\item[$\bullet$]
There is a monotone projection $\Pi:\TT\to\TT$ with $\Pi(0) = 0$ and  $\Pi([x_j,y_j]) = z_j$ which semiconjugates $\mathbf{m}_3$ to $\mathbf{m}_2$ on $\TT \sm \;]x_j,y_j[$. Thus both orbits $\{ x_i \}, \{ y_i \}$ realize the same cycle in $\sC_q$ as $\{ z_i \}$ does under $\mathbf{m}_2$. \vs 
\end{enumerate}

\begin{figure}[t]
\captionsetup{justification=justified}
\begin{overpic}[width=\textwidth]{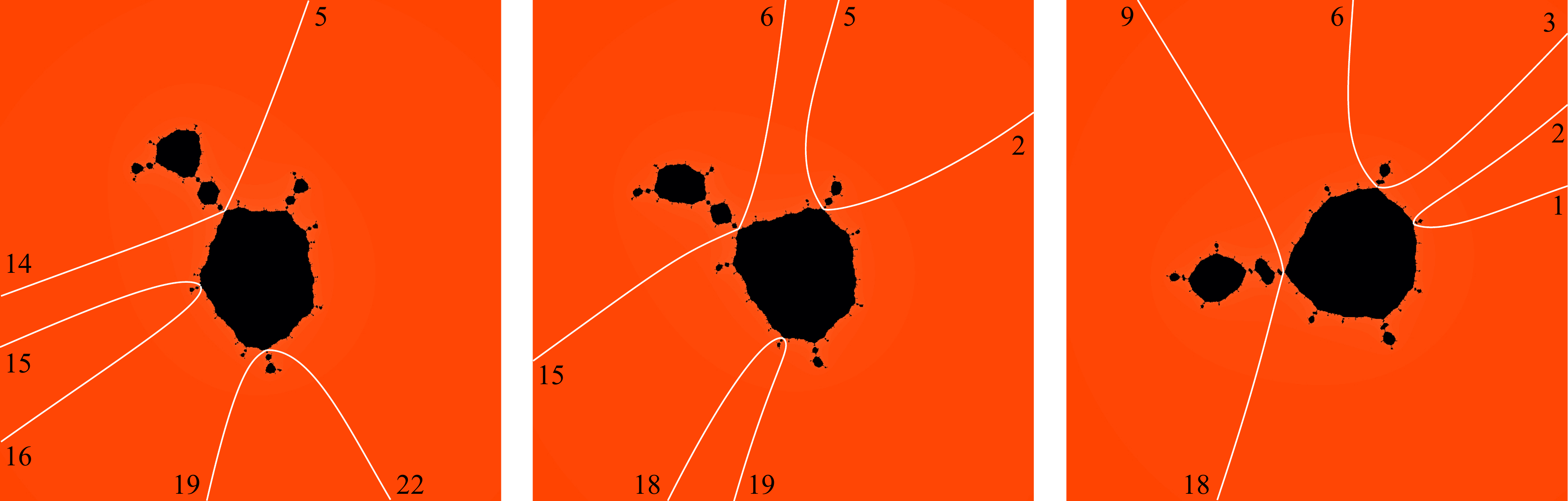}
\end{overpic}
\caption{The orbits $\OO_0 =\{ 14/26, 16/26, 22/26 \}, \ \OO_1 =\{ 5/26, 15/26, 19/26 \}, \ \OO_2 =\{ 2/26, 6/26, 18/26 \}$ and $\OO_3 = \{ 1/26, 3/26, 9/26 \}$
form all realizations of the cycle $\sigma=(1 \ 2 \ 3)$ under the tripling map $\mathbf{m}_3$. The pairs $\{ \OO_0, \OO_1 \}, \{ \OO_1, \OO_2 \}, \{ \OO_2, \OO_3 \}$ are compatible (the corresponding cubic polynomials and co-landing rays appear from left to right, with the angles shown in multiples of $1/26$).}  
\label{cub}
\end{figure}

\noindent
This was first proved by Milnor \cite{M} who used it to define the wake of each parabolic parameter on the boundary of the central hyperbolic component in the cubic slice $\{ z \mapsto a z^2 + z^3 \}_{a \in \CC}$. The resulting neighboring orbits are among the $q+1$ realizations under $\mathbf{m}_3$ of the same cycle in $\sC_q$ (the case $(d,a_q)=(1,0)$ or $(d,a_q)=(2,1)$ of the counting theorem). If we label these realizations as $\OO_0, \ldots, \OO_q$ where $\dep(\OO_j)=(j,q)$, then the neighboring orbits are precisely the pairs $\{ \OO_{j-1}, \OO_j \}$ for $1 \leq j \leq q$ (compare \figref{cub}). The dynamic rays with angles from these neighboring pairs define a new kind of co-landing pattern which is not seen in the quadratic family. \vs

The heuristic {\it principle of plenitude}, that a polynomial with given combinatorics must exist if it is not forbidden by some tangible obstruction, is generally accepted in holomorphic dynamics but of course has to be proved in each instance. A full understanding of which combinatorics of periodic orbits can exist and which combinatorics can conjoin to form a particular pattern was crucial to understanding the space of quadratic polynomials. Our analysis of general $q$-cycles in this paper can be viewed as a step toward a systematic treatment of similar questions for higher degree polynomials.  \vs

\noindent
{\it Acknowledgements.} We would like to thank IMS at Stony Brook for its hospitality during the conception of this paper. We are especially grateful to J.~Milnor for the inspiring conversations that motivated the present work. C.~L.~P. is supported by the Danish Council for Independent Research $|$ Natural Sciences via grant DFF 4181-00502. S.~Z. is partially supported by the Research Foundation of the City University of New York via grant TRADA-47-690. 

\section{The combinatorial type of a cycle}\label{CTC}

Let $\sS_q$ denote the group of all permutations of $q \geq 2$ objects. It will be convenient to think of $\sS_q$ as acting on the additive group $\ZZ/ q\ZZ$, although we often identify elements of $\ZZ/ q\ZZ$ with their unique representatives in $\{ 1,\ldots, q \}$. The collection of all $q$-cycles in $\sS_q$ is denoted by $\sC_q$. Following the tradition of group theory, we represent $\sigma \in \sC_q$ by the symbol
$$
( 1 \ \, \sigma(1) \ \, \sigma^2(1) \ \, \cdots \ \, \sigma^{q-1}(1) ). 
$$
The {\bit rotation group} $\sR_q$ is the cyclic subgroup of $\sS_q$ generated by the $q$-cycle 
$$
\rho = (1 \ \, 2 \ \, \cdots \ \, q). 
$$
Elements of $\sR_q \cap \sC_q$ are called {\bit rotation cycles}. Thus, $\sigma \in \sC_q$ is a rotation cycle if and only if $\sigma = \rho^p$ for some integer $1 \leq p < q$ relatively prime to $q$. The reduced fraction $p/q$ is called the {\bit rotation number} of $\rho^p$. \vs

The rotation group $\sR_q$ acts on $\sC_q$ by conjugation. We refer to each orbit of this action as a {\bit combinatorial type} in $\sC_q$. The combinatorial type of a $q$-cycle $\sigma$ is denoted by $[\sigma]$. It is easy to see that $\sigma$ is a rotation cycle if and only if $[\sigma]$ consists of $\sigma$ only. In fact, if $\rho \sigma \rho^{-1} = \sigma$, then $\sigma=\rho^p$ where $p=\sigma(q)$ (mod $q$). \vs

The combinatorial type of $\sigma \in \sC_q$ can be explicitly described as follows.
Let 
$$
\sG= \big\{ \rho^i : \rho^i \sigma \rho^{-i}=\sigma \big\}
$$ 
be the stabilizer of $\sigma$ under the action of $\sR_q$. We call the order of $\sG$ the {\bit symmetry order} of $\sigma$ and denote it by $\sym(\sigma)$. Evidently $\sG$ is generated by the power $\rho^r$ where $r=q/\sym(\sigma)$, and the combinatorial type of $\sigma$ is the $r$-element set 
$$
[\sigma]= \big\{ \sigma, \rho \sigma \rho^{-1}, \ldots, \rho^{r-1} \sigma \rho^{-(r-1)} \big\}.
$$
Since $\sym(\rho \sigma \rho^{-1})=\sym(\sigma)$, we can define the symmetry order of a combinatorial type unambiguously as that of any cycle representing it:
$$
\sym([\sigma])=\sym(\sigma). 
$$
 
It can be shown that for every $q \geq 5$ and every divisor $s$ of $q$ there is a $\sigma \in \sC_q$ with $\sym(\sigma)=s$. Note that for $q=2,3$ there are no combinatorial types of symmetry order $1$, while for $q=4$ there is no combinatorial type of symmetry order $2$.\footnote{For a precise count of the number of cycles in $\sC_q$ with a given symmetry order and also those with a given descent number (to be defined in the next section), see \cite{Z2}.} \vs

Of the $(q-1)!$ elements of $\sC_q$, precisely $\varphi(q)$ are rotation cycles, where $\varphi$ is Euler's totient function. If $\sigma_1, \ldots, \sigma_n$ represent the distinct combinatorial types in $\sC_q$, then
$$
(q-1)! = \sum_{\sigma_i \in \sR_q } \# [ \sigma_i] + \sum_{\sigma_i \notin \sR_q } \# [ \sigma_i] = \varphi(q)+ \sum_{\sigma_i \notin \sR_q } \# [ \sigma_i].
$$
In the special case where $q$ is a prime number, each $\# [\sigma_i]$ in the far right sum is $q$ and the number of distinct combinatorial types in $\sC_q$ is given by 
$$
n = (q-1) + \frac{(q-1)!-(q-1)}{q} = \frac{(q-1)!+(q-1)^2}{q}.
$$

\begin{example}\label{cyc5}
The $4!=24$ cycles in $\sC_5$ fall into $(4!+4^2)/5=8$ distinct combinatorial types. The $4$ rotation cycles 
\begin{align*}
\rho & =(1 \ 2 \ 3 \ 4 \ 5) \\
\rho^2 & = (1 \ 3 \ 5 \ 2 \ 4) \\
\rho^3 & = (1 \ 4 \ 2 \ 5 \ 3) \\
\rho^4 & = (1 \ 5 \ 4 \ 3 \ 2) 
\end{align*}
(with rotation numbers $1/5, 2/5, 3/5, 4/5$ respectively) form $4$ distinct combinatorial types. The remaining $20$ cycles have symmetry order $1$, so they fall into $4$ combinatorial types each containing $5$ elements. These types are represented by
\begin{align*}
\nu & =(1 \ 2 \ 3 \ 5 \ 4) \\
\nu^{-1} & = (1 \ 4 \ 5 \ 3 \ 2) \\
\sigma & = (1 \ 2 \ 4 \ 5 \ 3) \\
\sigma^{-1} & = (1 \ 3 \ 5 \ 4 \ 2) 
\end{align*}
Compare \figref{c5}.  
\end{example}

\begin{figure}[t]
\begin{overpic}[width=\textwidth]{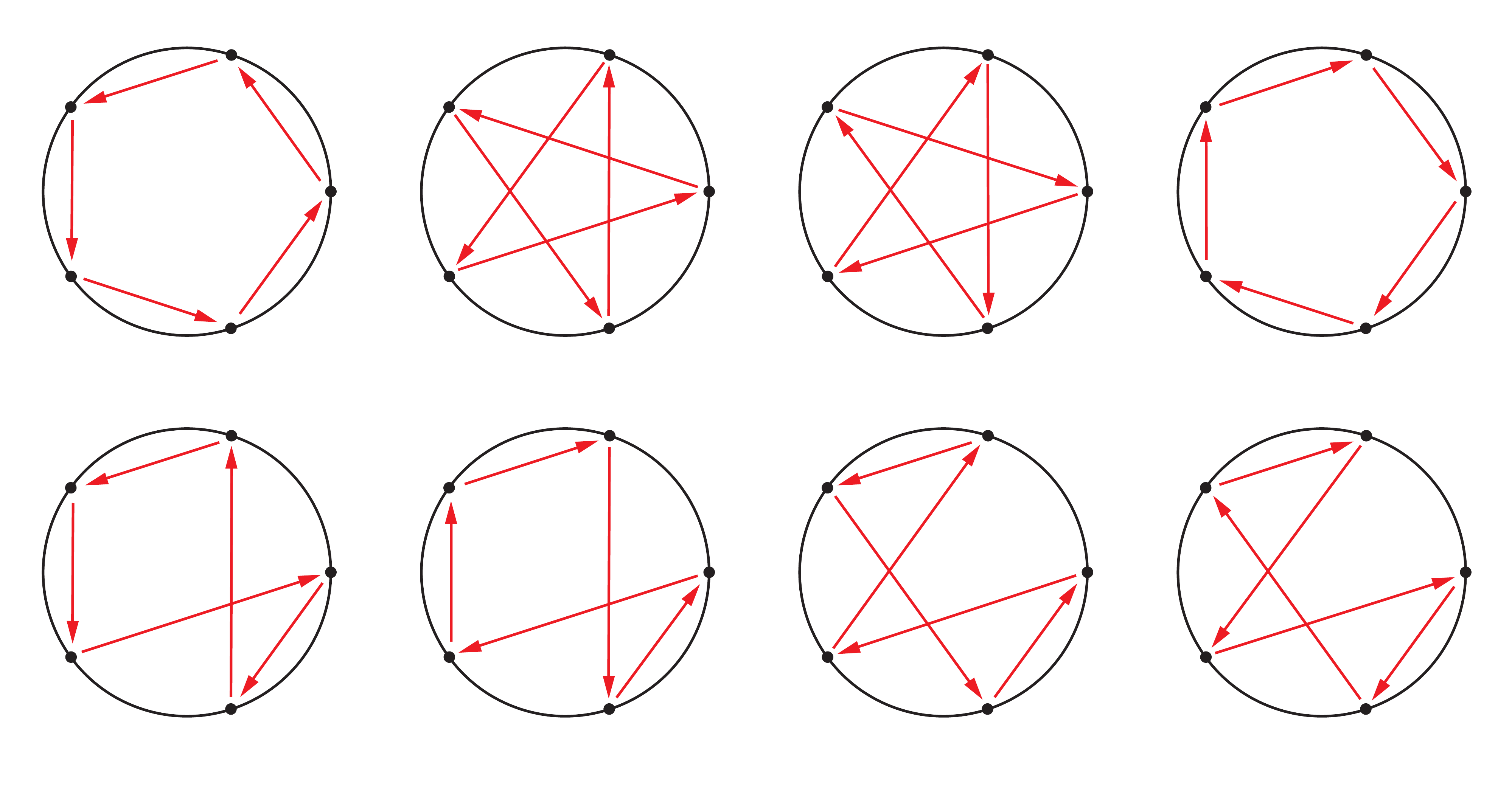}
\put (11.5,28) {\small $\rho$}
\put (37,28) {\small $\rho^2$}
\put (62,28) {\small $\rho^3$}
\put (87,28) {\small $\rho^4$}
\put (11,2.5) {\small $\nu$}
\put (36,2.5) {\small $\nu^{-1}$}
\put (62,2.5) {\small $\sigma$}
\put (86,2.5) {\small $\sigma^{-1}$}
\end{overpic}
\caption{\sl The representative cycles of the $8$ combinatorial types in $\sC_5$ described in \exref{cyc5}.}  
\label{c5}
\end{figure}

\begin{definition}\label{desn}
A permutation $\sigma \in \sS_q$ has a {\bit descent} at $i$ if $\sigma(i)>\sigma(i+1)$. The total number of such $i$ is called the {\bit descent number} of $\sigma$ and is denoted by $\des(\sigma)$. Thus,
$$
\des(\sigma) = \# \big\{ i \in \ZZ/ q\ZZ  : \sigma(i) > \sigma(i+1) \big\}. 
$$
\end{definition}

Notice that the descent number is defined using the {\it linear} order on the representatives $\{ 1, \ldots, q \}$ in $\ZZ/ q\ZZ$. \lemref{degp} below will show that any other choice of $q$ consecutive representatives would give the same count. For a topological interpretation of the descent number as the degree of an associated  covering map, see \secref{realcyc}. 

\begin{remark}
Our definition of descent number is slightly different from what is commonly used by combinatorists, as they often exclude $i=q$ from the descent count if $\sigma(q)>\sigma(1)$ (compare \cite{GKP} and \cite{S}). We have adopted the above definition to make the descent number rotationally invariant. 
\end{remark}

\begin{example}
The $5$-cycle $\sigma=(1 \ 2 \ 4 \ 5 \ 3)$, written in the classical notation as
$$
\sigma = \begin{pmatrix} 1 & 2 & 3 & 4 & 5 \\ 2 & 4 & 1 & 5 & 3 \end{pmatrix},
$$
has descents at $i=2$, $i=4$ and $i=5$, so $\des(\sigma)=3$. The eight representative cycles in $\sC_5$ described in \exref{cyc5} have the following descent numbers:
\begin{align*}
& \des(\rho)=\des(\rho^2)=\des(\rho^3)=\des(\rho^4)=1, \\
& \des(\nu)=\des(\nu^{-1})=2, \\
& \des(\sigma)=\des(\sigma^{-1})=3.
\end{align*}
\end{example} 

\begin{lemma}\label{degp}
For every $\sigma \in \sS_q$,
$$
\des(\sigma \rho^{-1})= \des(\rho \sigma)=\des(\rho \sigma \rho^{-1})=\des(\sigma).
$$ 
\end{lemma}

This allows us to define the descent number of a combinatorial type in $\sC_q$ unambiguously as that of any cycle representing it:
$$
\des([\sigma]) = \des(\sigma).
$$

\begin{proof}
Evidently $\sigma$ has a descent at $i$ if and only if $\sigma \rho^{-1}$ has a descent at $i+1$, so the equality $\des(\sigma \rho^{-1})=\des(\sigma)$ holds. \vs

Next, suppose $\sigma$ has a descent at $i$. Then $\rho \sigma$ has a descent at $i$ or $i-1$ according as $\sigma(i) \neq q$ or $\sigma(i)=q$. This shows $\des(\rho \sigma) \geq \des (\sigma)$. Similarly, if $\rho \sigma$ has a descent at $i$, then $\sigma$ has a descent at $i$ or $i+1$ according as $\rho \sigma(i+1) \neq 1$ or $\rho \sigma(i+1)=1$. This proves the reverse inequality $\des(\rho \sigma) \leq \des (\sigma)$, so $\des(\rho \sigma)=\des(\sigma)$ holds as well. \vs 

Finally, the identity $\des(\rho \sigma \rho^{-1})=\des(\sigma)$ follows from the two proved above.  
\end{proof} 
  
\begin{example}\label{rc=1}
$\sigma \in \sC_q$ is a rotation cycle if and only if $\des(\sigma)=1$. Here is why: If $\sigma$ is a rotation cycle, then $\sigma=\rho^p$ for some $p$, so by \lemref{degp}, $\des(\sigma)=\des(\rho)=1$. Conversely, suppose $\des(\sigma)=1$ and that the unique descent of $\sigma$ occurs at $i=q-p$ for some $p$. Then, $\nu = \sigma \rho^{-p}$ has a unique descent at $i=q$, so $\nu(1)<\nu(2)<\cdots<\nu(q)$. This can happen only if $\nu(i)=i$ for all $i$. It follows that $\nu = \text{id}$ and $\sigma = \rho^p$.      
\end{example}

\begin{example}
This is inspired by a suggestion of D. Thurston. Let $p$ be an odd prime and $d$ be a generator of the multiplicative group $(\ZZ/p\ZZ)^\ast$ of integers $\{ 1, \ldots, p-1 \}$ modulo $p$ (there are $\varphi(p-1)$ such generators). Consider the permutation $\sigma \in \sC_{p-1}$ defined by multiplication by $d$:
$$
\sigma(i)=di \ (\text{mod} \ p). 
$$
If $1 \leq \sigma(i) \leq p-d-1$, then $\sigma(i+1)=\sigma(i)+d>\sigma(i)$ so $i$ is not a descent of $\sigma$. If $p-d+1 \leq \sigma(i) \leq p-1$, then $\sigma(i+1)=\sigma(i)+d-p<\sigma(i)$, so $i$ is a descent of $\sigma$. Finally, if $\sigma(i)=p-d$, then necessarily $i=p-1$, and $i$ is a descent of $\sigma$ if and only if $2d<p$. It follows that 
$$
\des(\sigma)= \begin{cases} d & \quad \text{if} \ 2d<p \\ d-1 & \quad \text{if} \ 2d>p.  \end{cases}
$$             
For example, multiplication by $d=3$ modulo $p=7$ gives the $6$-cycle $\sigma=(1 \ 3 \ 2 \ 6 \ 4 \ 5)$ with $\des(\sigma)=3$, while multiplication by $d^{-1}=5$ modulo $7$ gives the inverse $6$-cycle $\sigma^{-1}=(1 \ 5 \ 4 \ 6 \ 2 \ 3)$ with $\des(\sigma^{-1})=4$.  
\end{example}

It is clear that every $2$- or $3$-cycle is a rotation cycle and therefore has descent number $1$. More generally, we have the following  

\begin{theorem}
If $\sigma \in \sC_q$ and $q \geq 3$, then $1 \leq \des(\sigma) \leq q-2$.   
\end{theorem}

\begin{proof}
Evidently $1 \leq \des(\sigma) \leq q-1$, so we only need to rule out the possibility $\des(\sigma)=q-1$. Assume by way of contradiction that this is the case, so $\sigma$ has $q-1$ descents and a unique ``ascent'' at some $i=n$. Then the permutation $\nu=\sigma \rho^{n-q}$ has a unique ascent at $i=q$, so $\nu(1)>\nu(2)>\cdots>\nu(q)$. This implies $\nu(i)=q-i+1$, or equivalently $\sigma(i)=n-i+1$ for all $i$. If $n=2j-1$ is odd, it follows that $j$ is a fixed point of $\sigma$, which is impossible. If $n=2j$ is even, it follows that $\sigma(j)=j+1$ and $\sigma(j+1)=j$. This too is impossible since $\sigma$, being a cycle of length $\geq 3$, cannot contain a $2$-cycle.             
\end{proof}

\section{Covering maps and fixed points}\label{sec:CMFP}

We will be working with the additive model of the circle as the quotient $\RR/\ZZ$ under the natural projection $\pi: \RR \to \RR/\ZZ$. Let us say that $3$ or more distinct points $x_1,x_2, \ldots, x_n$ on the circle are in {\bit positive cyclic order} if they have representatives $t_i \in \pi^{-1}(x_i)$ on the real line such that $t_1<t_2<\cdots<t_n<t_1+1$. For a distinct pair of points $a,b$ on the circle, the open interval $(a,b)$ consists of all $x \in \RR/\ZZ$ for which $a,x,b$ are in positive cyclic order. We define the intervals $[a,b],(a,b],[a,b)$ by adding suitable boundary points to $(a,b)$. When there is no danger of confusion, we may identify a point $x$ on the circle with its unique representative $t \in \pi^{-1}(x)$ which satisfies $0 \leq t <1$. \vs
     
A continuous map $f: \RR/\ZZ \to \RR/\ZZ$ is a covering map of degree $k \geq 1$ if it lifts under $\pi$ to a homeomorphism $F:\RR \to \RR$ which satisfies $F(t+1)=F(t)+k$ for all $t$. The fixed points of $f$ are in one-to-one correspondence with the points $0 \leq u <1$ for which $F(u)-u \in \ZZ$. We call the fixed point {\bit topologically repelling} if $t \mapsto F(t)-t$ is strictly increasing in some neighborhood of $u$. This condition holds, for example, if $F$ is $C^1$-smooth near $u$ and $F'(u)>1$. 

\begin{lemma}\label{fpcount}
A degree $k$ covering map $f: \RR/\ZZ \to \RR/\ZZ$ has at least $k-1$ fixed points. If all fixed points of $f$ are topologically repelling, then $f$ has precisely $k-1$ fixed points.
\end{lemma}

\begin{proof}
Take any lift $F: \RR \to \RR$ of $f$ and consider the function $T(t)=F(t)-t$ for $0 \leq t \leq 1$ which satisfies $T(1)=T(0)+k-1$. By the intermediate value theorem, for each of the $k-1$ integers $j$ satisfying $T(0) \leq j < T(1)$, the equation $T(t)=j$ has at least one solution. This shows that $f$ has at least $k-1$ fixed points on the circle. If all fixed points of $f$ are topologically repelling, then $T$ is increasing in some neighborhood of $t$ whenever $T(t)$ is an integer. It easily follows that the equation $T(t)=j$ has no solution if $j<T(0)$ or $j>T(1)$, and that for $T(0) \leq j < T(1)$ the solution to the equation $T(t)=j$ is unique.    
\end{proof} 

\begin{remark}
An alternative route to the above count in the topologically repelling case is the Lefschetz fixed point formula
$$
\sum_{f(x)=x} \operatorname{index}(f,x) = \sum_{i=0}^1 (-1)^i \operatorname{tr}\big( f_\ast: H_i(\RR/\ZZ) \to H_i(\RR/\ZZ) \big). 
$$
Since every topologically repelling fixed point has index $-1$, this formula reduces to
$$
\sum_{f(x)=x} (-1) = 1-k,
$$    
which shows that the number of fixed points of $f$ is $k-1$.
\end{remark}

Let $f,g: \RR/\ZZ \to \RR/\ZZ$ be covering maps and $a,b$ be distinct points on the circle with $f(a)=g(a)$ and $f(b)=g(b)$. Take representatives $\hat{a} \in \pi^{-1}(a)$ and $\hat{b} \in \pi^{-1}(b)$ with $\hat{a}<\hat{b}<\hat{a}+1$. Let $F: \RR \to \RR$ be any lift of $f$ and choose the lift $G: \RR \to \RR$ of $g$ such that $G(\hat{a})=F(\hat{a})$. Then $G(\hat{b})=F(\hat{b})+p$ for some integer $p$ which is independent of all the choices. By switching the roles of $f,g$ if necessary, we may assume $p \geq 0$. Intuitively, this means that the image of $[a,b]$ under $g$ starts at $g(a)=f(a)$ and winds counterclockwise around the circle $p$ full times more than $f$ before it ends at $g(b)=f(b)$. It will be convenient to express this situation by saying that $g$ is a {\bit $p$-winding} of $f$ on $[a,b]$. Evidently, any two $p$-windings of $f$ on $[a,b]$ are homotopic rel $\{ a, b \}$. \vs

One can easily construct $p$-windings as follows: Start with a covering map $f$ of degree $k$, an interval $[a,b]$ on the circle and an integer $p \geq 0$. Choose representatives $\hat{a},\hat{b}$ of $a,b$ with $\hat{a}<\hat{b}<\hat{a}+1$ and a lift $F$ of $f$. Modify $F$ to a new map $G$ by setting
$$
G(t)= \begin{cases} F(t)+ \left(\dfrac{t-\hat{a}}{\hat{b}-\hat{a}} \right) p & \qquad \hat{a} \leq t \leq \hat{b} \vs \\
F(t)+p & \qquad \hat{b} \leq t \leq \hat{a}+1 \end{cases} 
$$
and then extend $G$ to the real line by the relation $G(t+1)=G(t)+k+p$. The induced circle map $g: \RR/\ZZ \to \RR/\ZZ$ is clearly a covering map of degree $k+p$ which is a $p$-winding of $f$ on $[a,b]$. \vs

If $f: \RR/\ZZ \to \RR/\ZZ$ is a covering that maps $[a,b]$ homeomorphically onto $[f(a),f(b)]$, then any covering $g: \RR/\ZZ \to \RR/\ZZ$ with the same values as $f$ on the boundary $\{ a, b \}$ must be a $p$-winding of $f$ on $[a,b]$ for some $p \geq 0$. Under such $g$, every point in $[g(a),g(b)]=[f(a),f(b)]$ will have $p+1$ preimages in $[a,b]$.   

\begin{lemma}\label{fp+k}
Suppose $g$ is a $p$-winding of $f$ on $[a,b]$ for some $p \geq 0$. If all fixed points of $f$ and $g$ in $[a,b]$ are topologically repelling, then $g$ has $p$ fixed points more than $f$ in $[a,b]$.   
\end{lemma}

\begin{proof}
The argument is similar to the proof of \lemref{fpcount}. Take representatives $\hat{a},\hat{b}$ with $\hat{a}<\hat{b}<\hat{a}+1$ and lifts $F,G$ with $G(\hat{a})=F(\hat{a})$ and $G(\hat{b})=F(\hat{b})+p$. Set $T_1(t)=F(t)-t$ and $T_2(t)=G(t)-t$. If the fixed points of $f$ in $[a,b]$ are topologically repelling, the equation $T_1(t)=j \in \ZZ$ has a solution $\hat{a} \leq t \leq \hat{b}$ if and only if $T_1(\hat{a}) \leq j \leq T_1(\hat{b})$ and this solution is unique. Therefore, the number of fixed points of $f$ in $[a,b]$ is equal to the number of integers $j$ satisfying $T_1(\hat{a}) \leq j \leq T_1(\hat{b})$. A similar argument shows that the number of fixed points of $g$ in $[a,b]$ is equal to the number of integers $j$ satisfying $T_2(\hat{a}) \leq j \leq T_2(\hat{b})$. This proves the theorem since $T_2(\hat{a})=T_1(\hat{a})$ and $T_2(\hat{b})=T_1(\hat{b})+p$.        
\end{proof}

\section{Realizations of cycles}\label{realcyc}

We now turn to the problem of realizing cycles in $\sC_q$ as periodic orbits of covering maps of the circle. {\it By convention, we always assume that a period $q$ orbit $\{ x_1, \ldots, x_q \}$ on the circle is labeled so that the points $0,x_1,\ldots,x_q$ are in positive cyclic order. The indices of the $x_i$ are always taken modulo $q$.} 

\begin{definition}
A {\bit realization} of the cycle $\sigma \in \sC_q$ is a pair $(f,\OO)$, where $f: \RR/\ZZ \to \RR/\ZZ$ is a covering map, $\OO=\{ x_1, \ldots, x_q \}$ is a period $q$ orbit of $f$, and $f(x_i)=x_{\sigma(i)}$ for all $i$. By a slight abuse of language, sometimes we refer to $\OO$ itself as a realization of $\sigma$ {\bit under} $f$. By the {\bit degree} of the realization $(f,\OO)$ is meant the mapping degree of $f$. A realization of $\sigma$ of the lowest possible degree is called a {\bit minimal realization} of $\sigma$.     
\end{definition} 

The notion of a realization extends naturally to combinatorial types: $(f,\OO)$ is a realization of a combinatorial type $\tau$ in $\sC_q$ if it is a realization of some cycle $\sigma$ with $[\sigma]=\tau$. \vs

The following result (in a slightly different language) appears in \cite{Mc}: 

\begin{theorem}\label{lbound}
A cycle $\sigma \in \sC_q$ has a realization of degree $k$ if and only if $k \geq \des(\sigma)$.
\end{theorem}

Thus, minimal realizations of $\sigma$ exist and have degree equal to $\des(\sigma)$. 

\begin{proof}
A realization of degree $\des(\sigma)$ can be constructed as follows: Take any set $\OO=\{ x_1, \ldots, x_q \}$ with $0, x_1, \ldots, x_q$ in positive cyclic order, and define $f: \RR/\ZZ \to \RR/\ZZ$ by mapping each interval $[x_i,x_{i+1}]$ homeomorphically onto $[x_{\sigma(i)},x_{\sigma(i+1)}]$. Evidently $f$ is a covering map of the circle and $(f,\OO)$ is a realization of $\sigma$. To see that the degree of $f$ is $\des(\sigma)$, simply note that $\sigma(i)>\sigma(i+1)$ if and only if the image $f([x_i,x_{i+1}])$ contains $0$, so the number of descents of $\sigma$ is equal to the cardinality of $f^{-1}(0)$. \vs

To produce realizations of $\sigma$ of higher degrees, let $(f,\OO)$ be as above and let $g$ be a $p$-winding of $f$ on any of the intervals $[x_i,x_{i+1}]$. Then $(g,\OO)$ is a realization of $\sigma$ of degree $\des(\sigma)+p$. \vs      

Finally, let $(g, \OO)$ be any realization of $\sigma$. Let $\OO=\{ x_1, \ldots, x_q \}$, so $g(x_i)=x_{\sigma(i)}$ for all $i$. Define $f$ by mapping each interval $[x_i,x_{i+1}]$ homeomorphically onto $[x_{\sigma(i)},x_{\sigma(i+1)}]$, so $(f,\OO)$ is a realization of degree $\des(\sigma)$. For each $i$, there is an integer $p_i \geq 0$ such that $g$ is a $p_i$-winding of $f$ on $[x_i,x_{i+1}]$. It follows that the degree of $g$ is $\des(\sigma)+\sum p_i \geq \des(\sigma)$.     
\end{proof}

As seen from the above proof, there is a great deal of freedom in choosing the orbit and the covering map that define a realization of a given cycle. However, the construction is unique in the sense described below. Let us call two realizations $(g_0,\OO_0)$ and $(g_1,\OO_1)$ of $\sigma$ {\bit topologically equivalent} if there are orientation-preserving homeomorphisms $\psi_0,\psi_1: \RR/\ZZ \to \RR/\ZZ$, with $\psi_0(0)=\psi_1(0)=0$ and $\psi_0(\OO_0)=\psi_1(\OO_0)=\OO_1$, such that $\psi_0 \circ g_0 = g_1 \circ \psi_1$. In other words, the following diagram commutes: \vs

\begin{center}
\begin{tikzcd}[column sep=small]
(\RR/\ZZ, \OO_0) \arrow[d,swap,"\psi_1"] \arrow[rr,"g_0"] & & (\RR/\ZZ, \OO_0) \arrow[d,"\psi_0"] \\
(\RR/\ZZ, \OO_1) \arrow[rr,"g_1"] & & (\RR/\ZZ, \OO_1) 
\end{tikzcd} 
\end{center}

\vs
\noindent
Given a realization $(g,\OO)$ of $\sigma$ of degree $k$, take a minimal realization $(f,\OO)$ over the same orbit, so for each $i$ there is a $p_i \geq 0$ such that $g$ is a $p_i$-winding of $f$ on $[x_i,x_{i+1}]$. This associates to $(g,\OO)$ a non-negative integer vector $\bp(g,\OO)=(p_1, \ldots, p_q)$ with $\sum p_i = k-\des(\sigma)$. It is not hard to see that two realizations $(g_0,\OO_0)$ and $(g_1,\OO_1)$ of $\sigma$ are topologically equivalent if and only if $\bp(g_0,\OO_0)=\bp(g_1,\OO_1)$. In particular, when $k=\des(\sigma)$, the vectors $\bp(g_0,\OO_0)$ and $\bp(g_1,\OO_1)$ are both zero and we conclude that {\it any two minimal realizations of $\sigma$ are topologically equivalent}. 

\section{The transition matrix}\label{TMC}

Let us begin with a notational convention. The quotient $\ZZ/ q \ZZ$ on which our $q$-cycles act can be thought of as an additive subgroup of the circle by identifying $i \in \ZZ/ q\ZZ$ with $i/q \in \RR/\ZZ$. This identification allows us to define open, half-open, or closed intervals in $\ZZ/ q\ZZ$ as the intersection of the corresponding interval on the circle with $\ZZ/ q\ZZ$. For example, in $\ZZ/ 5\ZZ$,  
$$
[4,2]=\{ 4,5,1,2 \}, \ [4,2)=\{ 4,5,1 \}, \ (4,2]=\{ 5,1,2 \},  \ (4,2)=\{ 5,1 \}. 
$$

\begin{definition}
The {\bit transition matrix} of $\sigma \in \sC_q$ is the $q \times q$ matrix $A=[ a_{ij} ]$ defined by
$$
a_{ij} = \begin{cases} \ 1 & \quad \text{if} \ j \in [\sigma(i),\sigma(i+1)) \\ \ 0 & \quad \text{otherwise.} \end{cases} 
$$
\end{definition}

The algebraic definition of the transition matrix has a simple dynamical interpretation that justifies the terminology. Let $(f,\OO)$ be a minimal realization of $\sigma$, where $\OO=\{ x_1, \ldots, x_q \}$ and as usual $0, x_1, \ldots, x_q$ are in positive cyclic order. The orbit $\OO$ defines a partition of the circle into the intervals $I_1,\ldots,I_q$, where $I_i=[x_i,x_{i+1}]$. Since $f$ maps each $[x_i,x_{i+1}]$ homeomorphically onto $[x_{\sigma(i)},x_{\sigma(i+1)}]$, we have 
$$
f(I_i)= \bigcup_{j \in [\sigma(i),\sigma(i+1))} I_j \ \qquad \text{for all} \ i.
$$ 
It follows that the entries of the transition matrix $A = [a_{ij}]$ satisfy    
$$
a_{ij} = \begin{cases} \ 1 & \quad \text{if} \ f(I_i) \supset I_j \\ \ 0 & \quad \text{otherwise.} \end{cases} 
$$
Since $f$ is a covering map of degree $d=\des(\sigma)$, {\it every column of the transition matrix $A$ contains exactly $d$ entries of $1$}. Along the last column of $A$ these $d$ entries occur precisely at the descents of $\sigma$, that is, $a_{iq}=1$ if and only if $\sigma$ has a descent at $i$. The column stochastic matrix $(1/d)A$ describes a Markov chain with states $I_1, \ldots, I_q$, with the probability of going from $I_j$ to $I_i$ equal to $1/d$ if $I_j \subset f(I_i)$ and equal to $0$ otherwise (compare \exref{C5} and \figref{tmc}).   

\begin{lemma}\label{des}
There are precisely $d-1$ entries of $1$ along the main diagonal of $A$. 
\end{lemma} 

\begin{proof}
Construct a special minimal realization of $\sigma$ by choosing the orbit $\OO=\{ x_1, \ldots, x_q \}$ equidistributed over the circle so each interval $I_i=[x_i,x_{i+1}]$ has length $1/q$, and then letting $f$ map each $I_i$ {\it affinely} onto the union $\bigcup_{j \in [\sigma(i),\sigma(i+1))} I_j$. Thus, $f$ has constant derivative $\alpha_i$ in the interior of $I_i$, where $\alpha_i$ is the number of $1$'s along the $i$-th row of $A$. In particular, $\alpha_1, \ldots, \alpha_q$ are positive integers with $\sum_{i=1}^q \alpha_i=d q$. \vs

Let $u$ be a fixed point of $f$ and $I_i$ be the unique partition interval that contains $u$. Since the orbit points $x_i$ are not fixed, $u$ belongs to the interior of $I_i$, so the derivative $f'(u)=\alpha_i$ is defined and satisfies $f'(u)\geq 1$. If $f'(u)=1$, the entire interval $I_i$ would have to be fixed under $f$, which is impossible since its endpoints are not fixed. Thus $f'(u) \geq 2$ and $u$ is topologically repelling. It follows from \lemref{fpcount} that $f$ has precisely $d-1$ fixed points. \vs

To complete the proof, simply note that the piecewise affine map $f$ has at most one fixed point in each interval $I_i$, and that $I_i$ contains a fixed point if and only if $f(I_i) \supset I_i$.  
\end{proof}

\begin{definition}
Let $A=[a_{ij}]$ be the transition matrix of $\sigma \in \sC_q$ with $\des(\sigma)=d$. The {\bit signature} of $\sigma$ is the integer vector formed by the main diagonal entries of $A$:  
$$
\sig(\sigma)=(a_{11},\ldots,a_{qq}).
$$
By \lemref{des}, $d-1$ components of $\sig(\sigma)$ are $1$ and the rest are $0$. If $(f,\OO)$ is any realization of $\sigma$ (minimal or not), and if $I_1,\ldots,I_q$ are the corresponding partition intervals, then $I_i$ is called a {\bit marked interval} if $a_{ii}=1$. Notice that when $d=1$, $\sig(\sigma)$ is the zero vector and there are no marked intervals.   
\end{definition}

For simplicity we denote the components of the signature by $a_i$: $\sig(\sigma)=(a_1, \ldots, a_q)$. Unlike the descent number or symmetry order, the signature is not well-defined for a combinatorial type. In fact, 
\begin{equation}\label{sigcon}
\sig(\sigma)=(a_1, a_2, \ldots, a_q) \quad \Longrightarrow \quad \sig(\rho^{-1} \sigma \rho)=(a_2, \ldots, a_q, a_1). 
\end{equation}
If $\sigma$ has symmetry order $s$, then $\rho^{-r} \sigma \rho^r=\sigma$ where $r=q/s$, hence $\sig(\sigma)$ is $r$-periodic in the sense that 
\begin{equation}\label{r-per}
a_i=a_{i+r} \qquad \text{for all} \  i. 
\end{equation}

\begin{example}\label{C5}
Consider the $5$-cycle $\sigma=(1 \ 2 \ 4 \ 5 \ 3)$ with $\des(\sigma)=3$. A minimal realization $(f,\OO)$ induces the following action on the partition intervals: 
\begin{align*}
f(I_1) & = I_2 \cup I_3 \\
f(I_2) & = I_4 \cup I_5 \\
f(I_3) & = I_1 \cup I_2 \cup I_3 \cup I_4 \\
f(I_4) & = I_5 \cup I_1 \cup I_2 \\
f(I_5) & = I_3 \cup I_4 \cup I_5 \cup I_1.  
\end{align*} 
This gives the transition matrix 
$$
A= \begin{bmatrix}
0 & 1 & 1 & 0 & 0 \\
0 & 0 & 0 & 1 & 1 \\
1 & 1 & 1 & 1 & 0 \\
1 & 1 & 0 & 0 & 1 \\
1 & 0 & 1 & 1 & 1 \\
\end{bmatrix}
$$
with two entries of $1$ along the main diagonal, as claimed by \lemref{des}. It follows that $\sig(\sigma)=(0,0,1,0,1)$ and the marked intervals are $I_3,I_5$. The corresponding Markov chain is shown in \figref{tmc}. 
\end{example}

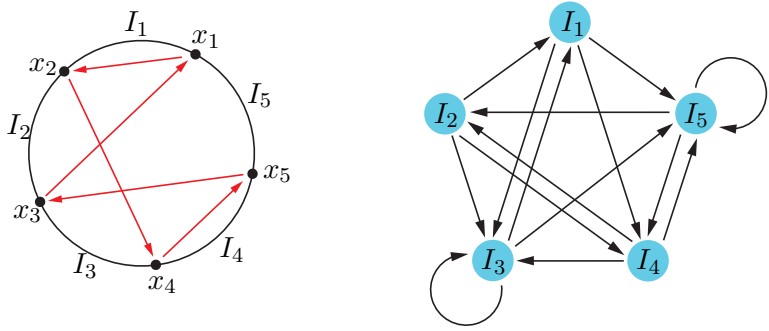
\begin{figure}[t]
\begin{overpic}[width=0.8\textwidth]{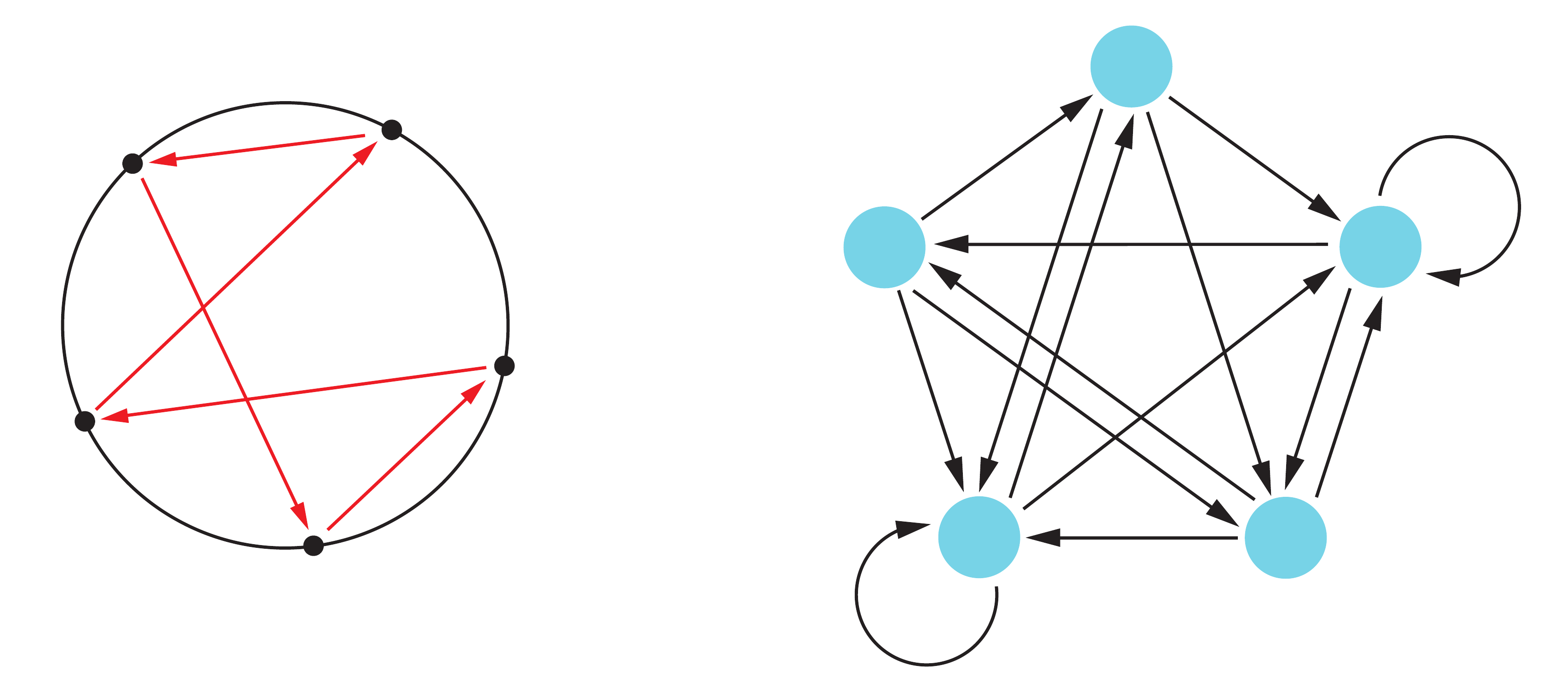}
\put (16,38.5) {\small $I_1$}
\put (1.2,25.5) {\small $I_2$}
\put (9.5,8) {\small $I_3$}
\put (28,10) {\small $I_4$}
\put (31.5,29.5) {\small $I_5$}
\put (24.5,37) {\small $x_1$}
\put (4,33.5) {\small $x_2$}
\put (2,15) {\small $x_3$}
\put (19,6) {\small $x_4$}
\put (33.5,20) {\small $x_5$}
\put (71,38.5) {\small $I_1$}
\put (55,27) {\small $I_2$}
\put (61,8.5) {\small $I_3$}
\put (80.5,8.5) {\small $I_4$}
\put (86.5,27) {\small $I_5$}
\end{overpic}
\caption{\sl Left: A realization of the $5$-cycle $\sigma=(1 \ 2 \ 4 \ 5 \ 3)$ with the marked intervals $I_3$ and $I_5$. Right: The Markov chain associated with $\sigma$ in which every transition has probability $1/\des(\sigma)=1/3$.}  
\label{tmc}
\end{figure}

Let $A$ be the transition matrix of $\sigma \in \sC_q$ and $(f,\OO)$ be a minimal realization of $\sigma$ with the partition intervals $I_1, \ldots, I_q$ as above. A straightforward induction shows that the $ij$-entry $a^{(n)}_{ij}$ of the power $A^n$ is the number of times the $n$-th iterated image $f^{\circ n}(I_i)$ covers $I_j$ or, equivalently, the number of connected components of $f^{-n}(I_j)$ in $I_i$. Since $\sigma$ is a $q$-cycle, for every pair $i,j$ there is an integer $1 \leq n \leq q$ such that $f^{\circ n}(x_i)=x_j$. It follows that $I_j \subset f^{\circ n}(I_i)$ and $a^{(n)}_{ij}>0$. This shows that the Markov chain associated with $\sigma$ is {\it transitive}: One can go from any state to any state in finitely many steps.\footnote{This property is also known as {\it irreducibility}; see \cite{F}.} \vs

This statement can be made sharper as follows. First assume $\des(\sigma)=1$, so $\sigma=\rho^p$ for some $p$ relatively prime to $q$ (\exref{rc=1}). In this case $a_{ij}=1$ if and only if $j=i+p \ (\operatorname{mod} \ q)$, and it follows by induction that
$$
a_{ij}^{(n)} = \begin{cases} \ 1 & \qquad \text{if} \ j=i+n p \ (\operatorname{mod} \ q) \\ \ 0 & \qquad \text{otherwise}. \end{cases}
$$
Setting $n=q$ then gives $A^q=\id$. The powers of $A$ are thus obtained by simultaneously shifting the rows of the identity matrix. \vs

The situation when $\des(\sigma) \geq 2$ is very different, as the entries of $A^n$ tend to grow rapidly with $n$. 
 
\begin{lemma}\label{regularity}
Let $A$ be the transition matrix of $\sigma \in \sC_q$ with $\des(\sigma)=d \geq 2$. Then the power $A^q$ has positive entries.   
\end{lemma} 

This means that the Markov chain associated with $\sigma$ is {\it regular} in the sense that one can go from any state to any state in exactly $q$ steps.

\begin{proof}
Consider a minimal realization $(f,\OO)$ of $\sigma$ and the corresponding partition intervals $I_1, \ldots , I_q$. We need to show that $f^{\circ q}(I_i)=\RR/\ZZ$ for every $i$. Since the iterate $f^{\circ q}$ fixes the endpoints of $I_i$, it must be a $p$-winding of the identity map on $I_i$ for some $p \geq 0$. If $p=0$, then $f^{\circ q}: I_i \to I_i$ is a homeomorphism, which implies that the orbit of $I_i$ under $f$ is $I_i \mapsto I_{\sigma(i)} \mapsto I_{\sigma^2(i)} \mapsto \cdots \mapsto I_{\sigma^{q-1}(i)} \mapsto I_i$. This gives $d=1$, contradicting our assumption. Thus $p \geq 1$ and we conclude that $f^{\circ q}(I_i)=\RR/\ZZ$.  
\end{proof}

\begin{question}
Does \lemref{regularity} hold if we replace $A^q$ with $A^{q-d+1}$?  
\end{question}

The importance of the above lemma is due to the following classical result of Perron and Frobenius: 

\begin{theorem}\label{PF0}
Let $S$ be a $q \times q$ column stochastic matrix with the property that some power of $S$ has positive entries. Then 
\begin{enumerate}
\item[(i)]
$S$ has a simple eigenvalue at $\lambda=1$ and the remaining eigenvalues in the open unit disk $\{ \lambda: |\lambda|<1 \}$. \vs
\item[(ii)]
The eigenspace corresponding to $\lambda=1$ is one-dimensional and generated by a unique probability vector $\bl=(\ell_1, \ldots, \ell_q)$ with $\ell_i>0$ for all $i$. \vs
\item[(iii)]
The power $S^n$ converges to the matrix with identical columns $\bl$ as $n \to \infty$.
\end{enumerate} 
\end{theorem}

See for example \cite{Ga}, \cite{F} or \cite{GS}. \vs

The following corollary will be used in the next section:

\begin{theorem}\label{PF}
Let $A$ be the transition matrix of $\sigma \in \sC_q$ with $\des(\sigma)=d \geq 2$. Then, there is a unique probability vector $\bl \in \RR^q$ such that $A \bl = d \bl$. Moreover, $\bl$ has positive components and satisfies 
\begin{equation}\label{ELL}
\bl = \lim_{n \to \infty} \frac{1}{d^n} \, A^n \bv
\end{equation}
for every probability vector $\bv \in \RR^q$.
\end{theorem}

\begin{proof}
This follows from \thmref{PF0} since by \lemref{regularity}, some power of the column stochastic matrix $S=(1/d)A$ has positive entries. 
\end{proof}

As we have seen, the transition matrix of $\sigma$ encodes the dynamics of partition intervals under any minimal realization. It will be useful for our purposes to also consider transition matrices that encode higher degree realizations. Let $A$ be the transition matrix of $\sigma \in \sC_q$ with $\des(\sigma)=d \geq 1$. Let $\bp=(p_1,\ldots,p_q) \neq {\bf 0}$ be a non-negative integer vector. Consider the $q \times q$ matrix $P$ with identical columns $\bp$, so the $i$-th row of $P$ is $(p_i,\ldots,p_i)$. The sum
$$
B=A+P
$$
is called the {\bit transition matrix of the pair $(\sigma, \bp)$}. Setting 
$$
k = d + \sum_{i=1}^q p_i > d, 
$$
we see that the sum of the entries in every column of $B$ is $k$, so $(1/k) B$ is a column stochastic matrix. \vs

The matrix $B$ encodes the dynamics of partition intervals under a degree $k$ realization $(g,\OO)$ of $\sigma$ which is determined uniquely up to topological equivalence (compare the discussion at the end of \S\ref{realcyc}). To see this, take a minimal realization $(f,\OO)$ of $\sigma$ with the partition intervals $I_1,\ldots,I_q$ and let $g: \RR/\ZZ \to \RR/\ZZ$ be any covering map which is a $p_i$-winding of $f$ on $I_i$ for every $i$. Evidently $(g,\OO)$ is a degree $k$ realization of $\sigma$ with $\bp(g,\OO)=\bp$. Moreover, the image $g(I_i)$ covers the intervals in $\bigcup_{j \in [\sigma(i),\sigma(i+1))} I_j$ precisely $p_i+1$ times and it covers the remaining partition intervals $p_i$ times. Thus, the $ij$-entry of $B$ records how many times $I_j$ is covered by $g(I_i)$. \vs

Since the entries of $P$ are non-negative, the powers of $B$ grow at least as fast as those of $A$, that is, $B^n \geq A^n$ for all $n$ entry-wise. Thus, if $d \geq 2$, \lemref{regularity} shows that $B^q$ has positive entries. This is also true if $d=1$ but needs justification. Take a degree $k$ realization $(g,\OO)$ encoded by $B$ as above. Since the $ij$-entry of $B^n$ is the number of times $g^{\circ n}(I_i)$ covers $I_j$, it suffices to show that $g^{\circ q}(I_i)=\RR/\ZZ$ for every $i$. This is trivial if $p_i>0$ since in this case $g(I_i)$ already covers the circle. Otherwise, take the smallest integer $1 \leq j \leq q-1$ such that $f^{\circ j}(I_i)$ is an interval $I_n$ with $p_n>0$ and note that $g^{j+1}(I_i) \supset g(f^{\circ j}(I_i)) = g(I_n) = \RR/\ZZ$. \vs

Applying \thmref{PF0} to the column stochastic matrix $S=(1/k)B$, we arrive at the following analog of \thmref{PF}:   

\begin{theorem}\label{PF1}
Let $B$ be the transition matrix of the pair $(\sigma,\bp)$, where $\sigma \in \sC_q$ and $\bp=(p_1,\ldots,p_q) \neq {\bf 0}$ is a non-negative integer vector. Set $k=\des(\sigma)+\sum p_i$. Then, there is a unique probability vector $\bl \in \RR^q$ such that $B \bl = k \bl$. Moreover, $\bl$ has positive components and satisfies
$$
\bl = \lim_{n \to \infty} \frac{1}{k^n} \, B^n \bv
$$
for every probability vector $\bv \in \RR^q$.
\end{theorem}  
   
\section{Realizations under $\mk$: The minimal case}\label{sec:min}

Let $\mk: \RR/\ZZ \to \RR/\ZZ$ denote the multiplication by $k$ map of the circle defined by  
$$
\mk(x) = k x \quad (\operatorname{mod} \ \ZZ).
$$
The central question of this work is whether a given combinatorial type $\tau$ in $\sC_q$ has a realization under $\mk$ and if it does, how many such realizations there are. By \thmref{lbound}, a realization of $\tau$ under $\mk$ can exist only if $k \geq \des(\tau)$. This section treats the minimal case where $k=\des(\tau) \geq 2$. The next section will address the more general case where $k>\des(\tau)\geq 1$. 

\begin{theorem}\label{main1}
A cycle $\sigma \in \sC_q$ with $\des(\sigma)=d\geq 2$ and $\sig(\sigma)=(a_1, \ldots, a_q)$ has a realization under $\md$ if and only if $a_q=1$. In this case, the realization is unique. 
\end{theorem}

Notice that the signature condition $a_q=1$ is equivalent to the inequality $\sigma(1)<\sigma(q)$ (in the linear order of $\{ 1, \ldots, q \}$) since both conditions mean $i=q$ is a descent of $\sigma$.

\begin{proof}
Necessity is clear: If $(\md,\OO)$ is a realization of $\sigma$, then it is a minimal realization of $\sigma$. The interval $I_q$ in the corresponding partition contains the fixed point $0$ of $\md$. As such, $\md(I_q) \supset I_q$, which means $I_q$ is a marked interval. Thus, $a_q=1$. \vs

Conversely, suppose $a_q=1$. We are looking for a period $q$ orbit $\OO=\{ x_1, \ldots, x_q \}$ of $\md$ such that $\md(x_i) = x_{\sigma(i)}$ for all $i$. Assume for a moment that such $\OO$ exists, let $I_i=[x_i,x_{i+1}]$, consider the lengths $\ell_i=|I_i|$, and form the probability vector  
$$
\bl=(\ell_1,\ldots, \ell_q)
$$
in $\RR^q$ which has strictly positive components. Since $\md$ maps $I_i$ homeomorphically onto $\bigcup_{j \in [\sigma(i),\sigma(i+1))} I_j$, we have
\begin{equation}\label{foo}
\sum_{j \in [\sigma(i),\sigma(i+1))} \ell_j = d \ell_i \qquad \text{for all} \ i. 
\end{equation}
These $q$ relations can be written as  
\begin{equation}\label{AL}
A \bl = d \bl,
\end{equation}
where $A$ is the transition matrix of $\sigma$. By \thmref{PF}, this equation has a unique solution $\bl$ which determines the lengths of the partition intervals $\{ I_i\} $, hence the orbit $\OO$ once we find $x_1$. \vs

To construct the orbit $\OO=\{ x_1, \ldots, x_q \}$, take the unique solution $\bl = (\ell_1, \ldots, \ell_q)$ of \eqref{AL} and define  
\begin{equation}\label{deforb}
\begin{cases}
x_1 = & \dfrac{1}{d-1}  \ds{\sum_{j \in [1,\sigma(1))} \ell_j} \vs \vs \\
x_i = & x_1 \ + \ \ds{\sum_{j \in [1,i)} \ell_j} \qquad \text{for} \ 2 \leq i \leq q.
\end{cases}
\end{equation}
Evidently $0< x_1 < x_2 < \cdots <x_q$. Since $a_q=1$, we have the linear order $1<\sigma(1)<\sigma(q)<q$, so $[1,\sigma(1))$ is a proper subset of $[\sigma(q),\sigma(1))$ that does not contain $q$. Hence by \eqref{deforb} and \eqref{foo}, 
$$
(d-1) x_1 = \sum_{j \in [1,\sigma(1))} \ell_j < \ \big( \hspace{-6mm} \sum_{j \in [\sigma(q),\sigma(1))} \ell_j \big) - \ell_q = (d-1) \ell_q. 
$$
This implies $x_1< \ell_q = x_1+1-x_q$ which shows $x_q<1$. Thus, the points $0,x_1,\ldots,x_q$ are in positive cyclic order on the circle. \vs 

To see that $(\md,\OO)$ is a realization of $\sigma$, note that by \eqref{deforb}  
$$
d x_1 = x_1 + \sum_{j \in [1,\sigma(1))} \ell_j = x_{\sigma(1)},
$$
and for $2 \leq i \leq q$,
\begin{align*}
d x_i & = d x_1 + \sum_{j \in [1,i)} d \ell_i \\
& =  x_1 + \sum_{j \in [1,\sigma(1))} \ell_j + \sum_{j \in [1,i)} \ \sum_{n \in [\sigma(j),\sigma(j+1))} \ell_n \qquad (\text{by} \ \eqref{foo}) \\
& = x_1 + \sum_{j \in [1,\sigma(i))} \ell_j \qquad (\operatorname{mod} \ \ZZ) \\
& = x_{\sigma(i)}.
\end{align*}
Thus, $\md(x_i) = x_{\sigma(i)}$ for every $i$, as required. \vs

For uniqueness, suppose $\OO' = \{ x'_1, \ldots, x'_q \}$ is another realization of $\sigma$ under $\md$. By what we have seen, the lengths $\ell'_i=x'_{i+1}-x'_i$ must be identical to the above $\ell_i$, so $\OO'=\OO$ follows once we check that $x'_1=x_1$. Now 
$$
(d-1) x'_1 = x'_{\sigma(1)} -x'_1 = \! \! \sum_{j \in [1,\sigma(1))} \ell'_j = \! \! \sum_{j \in [1,\sigma(1))} \ell_j = (d-1) x_1 \quad (\operatorname{mod} \ \ZZ),
$$
which means $x'_1$ and $x_1$ differ by an integer multiple of $1/(d-1)$. Since $0$ is the only fixed point of $\md$ in $(x'_q,x'_1)$ and $(x_q,x_1)$, both $x'_1$ and $x_1$ belong to $(0,1/(d-1))$, from which it follows that $x'_1=x_1$.
\end{proof} 

\begin{remark}
The realization orbit constructed above can be thought of as the unique fixed point of the following operator $L$ acting on the space of all ordered $q$-tuples on the circle: Given a point $\by= \{ y_1, \ldots, y_q \}$ in this space, form the probability vector $\bv=(v_1,\ldots,v_q)$ where $v_i:=y_{i+1}-y_i$, set $\bv'=(v'_1,\ldots,v'_q):=(1/d)A \bv$, and define $L(\by):= \{ y'_1,\ldots,y'_q \}$ by the equation \eqref{deforb}, replacing $x_i$ by $y'_i$ and $\ell_i$ by $v'_i$. It is easy to see using \eqref{ELL} that for any $\by$ the iterates $L^{\circ n}(\by)$ converge to the unique realization orbit $\bx=\{ x_1,\ldots,x_q \}$ as $n \to \infty$. Interpreted this way, the construction is reminiscent of other realization proofs in dynamics that invoke Thurston's idea of iteration on a suitable Teichm\"{u}ller space. 
\end{remark}

\begin{corollary}\label{shal}
Let $(\md,\OO)$ be the unique realization of $\sigma \in \sC_q$ with $d=\des(\sigma) \geq 2$ and $\sig(\sigma)=(a_1,\ldots,a_q=1)$. Then, each of the $d-1$ marked intervals of $\OO$ contains a single fixed point of $\md$. More precisely, if $i_1<\cdots<i_{d-1}=q$ are the indices $i$ for which $a_i=1$, then the marked interval $I_{i_j}$ contains the fixed point $j/(d-1)$. 
\end{corollary}

\begin{proof}
A marked interval $I$ is characterized by the property $\md(I) \supset I$, so it must contain a unique fixed point of $\md$. The claim follows immediately since there are $d-1$ marked intervals and $d-1$ fixed points.  
\end{proof}

\begin{theorem}\label{main2}
Every combinatorial type $\tau$ in $\sC_q$ with $\des(\tau)=d \geq 2$ and $\sym(\tau)=s$ has $(d-1)/s$ realizations under $\md$. 
\end{theorem} 

It follows in particular that $\sym(\tau)$ is always a divisor of $\des(\tau)-1$.  
 
\begin{proof}
Pick a representative cycle $\sigma$ for $\tau$ and let $\sig(\sigma)=(a_1,\ldots,a_q)$. Recall from \S \ref{CTC} that $\tau = \{ \rho^{-1} \sigma \rho, \ldots, \rho^{-r} \sigma \rho^{r}=\sigma \}$, where $r=q/s$. By \eqref{sigcon}, 
$$
\sig(\rho^{-j} \sigma \rho^j) = (a_{j+1},\ldots,a_q,a_1,\ldots,a_j).
$$
It follows from \thmref{main1} that the number of realizations of $\tau$ under $\md$ is equal to the number of integers $1 \leq j \leq r$ for which $a_j=1$. But $\sig(\sigma)$ is $r$-periodic in the sense of \eqref{r-per}, so the number of $1$'s among $a_1, \ldots, a_r$ is $(d-1)/s$.   
\end{proof}

\begin{remark}
Here is an alternate formulation of the above proof: Pick $\sigma$ and let $\sig(\sigma)=(a_1,\ldots,a_q)$ as before. The stabilizer $\sG$ of $\sigma$ in $\sR_q$ is the cyclic group of order $s$ generated by $\rho^r$, where $r=q/s$. Consider the set $\Delta= \{ \rho^{i_1}, \ldots, \rho^{i_{d-1}} \} \subset \sR_q$, where $i_1<\cdots<i_{d-1}$ are the indices $i$ for which $a_i=1$. By \eqref{sigcon}, for each $\alpha \in \Delta$, the conjugate cycle $\alpha^{-1} \sigma \alpha$ has its signature ending in $1$, so by \thmref{main1} it has a unique realization under $\md$. Moreover, two such conjugates $\alpha^{-1} \sigma \alpha$ and $\beta^{-1} \sigma \beta$ coincide if and only if $\alpha \beta^{-1} \in \sG$. Thus, the number of realizations of $\tau$ under $\md$ is equal to the number of distinct cosets $\alpha \sG$ for $\alpha \in \Delta$. Since $\sigma$ commutes with $\sG$, we have $\Delta \sG = \Delta$, which means $\Delta$ is a disjoint union of the cosets of $\sG$. It follows that the number of such cosets is $\# \Delta / \# \, \sG =(d-1)/s$.
\end{remark}

The $(d-1)/s$ realizations of $\tau$ in the above theorem can be described as rotated copies of a single periodic orbit. To see this, take a representative cycle $\sigma$ of $\tau$ with $\sig(\sigma)=(a_1,\ldots,a_q=1)$ and let $\OO$ be the unique realization of it under $\md$ as given by \thmref{main1}. Since the rigid rotation $x \mapsto x-1/(d-1) \ (\operatorname{mod} \ \ZZ)$ commutes with $\md$, the $d-1$ rotated copies 
$$
\OO_j=\OO-\frac{j}{d-1} \qquad 1 \leq j \leq d-1
$$ 
are all period $q$ orbits under $\md$ and give realizations of the combinatorial type $\tau$. Using \corref{shal}, it is easy to check that $\OO_j$ is the realization of the conjugate cycle $\rho^{-i_j} \sigma \rho^{i_j}$ (as before, $i_1<\cdots<i_{d-1}=q$ are the indices $i$ for which $a_i=1$). Thus, the uniqueness part of \thmref{main1} shows that $\OO_j=\OO_k$ if and only if $\rho^{i_j-i_k}$ belongs to the stabilizer of $\sigma$, or equivalently, $i_j = i_k \ (\operatorname{mod} r)$ where $r=q/s$. By the proof of \thmref{main2} there are $n=(d-1)/s$ residue classes modulo $r$ in $\{ i_1, \ldots, i_{d-1} \}$, each containing $s$ elements. In fact, the $r$-periodicity of $\sig(\sigma)$ shows that $1 \leq i_1 < \cdots < i_n =r$, from which it follows that these $n$ residue classes are of the form $i_j, i_j+r, \ldots, i_j+(s-1)r$, where $1 \leq j \leq n$. Thus, $\OO_1, \ldots, \OO_n=\OO$ are the distinct realizations of $\tau$ under $\md$. This proves the following  

\begin{theorem} \label{rotcop}
Let $\tau$ be a combinatorial type in $\sC_q$ with $\des(\tau)=d \geq 2$ and $\sym(\tau)=s$. Take a representative cycle $\sigma$ for $\tau$ with $\sig(\sigma)=(a_1,\ldots,a_q=1)$, and let $\OO$ be the unique realization of $\sigma$ under $\md$. Then, the rotated copies
$$
\OO-\frac{j}{d-1} \qquad \text{for} \ \ \ 0 \leq j < \frac{d-1}{s}
$$
form all the realizations of $\tau$ under $\md$. Furthermore, each of these $(d-1)/s$ orbits is invariant under the rigid rotation $x \mapsto x-1/s \ \ (\operatorname{mod} \ \ZZ)$.   
\end{theorem}

\begin{corollary}
If $q$ is a prime number, every combinatorial type $\tau$ in $\sC_q$ with $\des(\tau)=d \geq 2$ has $d-1$ realizations under $\md$ which are rotated copies of one another under $x \mapsto x-1/(d-1)  \ (\operatorname{mod} \ \ZZ)$. 
\end{corollary}

The following three examples illustrate the above results. It will be convenient here and elsewhere to use the simpler notation 
$$
\frac{1}{b} \, \big\{ a_1, \ldots, a_q \big\} \qquad \text{for} \qquad \Big\{ \frac{a_1}{b}, \ldots, \frac{a_q}{b} \Big\} 
$$  
when dealing with realization orbits that consist of rational numbers with the same denominator.
 
\begin{example}\label{sinu}
Let $\tau$ be the combinatorial type in $\sC_5$ represented by the cycle $\sigma=(1 \ 2 \ 4 \ 5 \ 3)$ of \exref{C5} with $\sig(\sigma)=(0,0,1,0,1)$. Here $\des(\tau)=3$ and $\sym(\tau)=1$. By \thmref{main2}, there are two realizations of $\tau$ under $\mathbf{m}_3$, corresponding to the representative cycles $\sigma$ and $\nu = \rho^{-3} \sigma \rho^3 = (1 \ 2 \ 5 \ 3 \ 4)$. For $\sigma$, the transition matrix $A$ and the unique probability vector solution of $A \bl = 3 \bl$ are\footnote{The computation of $\bl$ here and in the subsequent examples is most conveniently performed by rounding off the entries of the matrix $(d^q-1)d^{-n}A^n$ for a large enough $n$. It is easy to see using \eqref{ELL} that the resulting integer matrix has identical columns, each being $(d^q-1) \bl$.} 
$$
A = \begin{bmatrix}
0 & 1 & 1 & 0 & 0 \\
0 & 0 & 0 & 1 & 1 \\
1 & 1 & 1 & 1 & 0 \\
1 & 1 & 0 & 0 & 1 \\
1 & 0 & 1 & 1 & 1 \\
\end{bmatrix}
\qquad \text{and} \qquad \bl = \frac{1}{3^5-1} \, \big( 32, 38, 58, 46, 68 \big).
$$
Using the relations \eqref{deforb}, we obtain the realization 
$$
\OO = \frac{1}{121} \ \big\{ 8, 24, 43, 72, 95 \big\}.
$$
Rotating $\OO$ under $x \mapsto x-1/2 \ (\operatorname{mod} \ \ZZ)$ then gives the realization $\OO'$ of $\nu$:
$$
\OO' = \OO - \frac{1}{2} = \frac{1}{242} \ \big\{ 23, 69, 137, 169, 207 \big\}
$$
(compare \figref{re1}). 
\end{example}

\begin{figure}[t]
\begin{overpic}[width=0.8\textwidth]{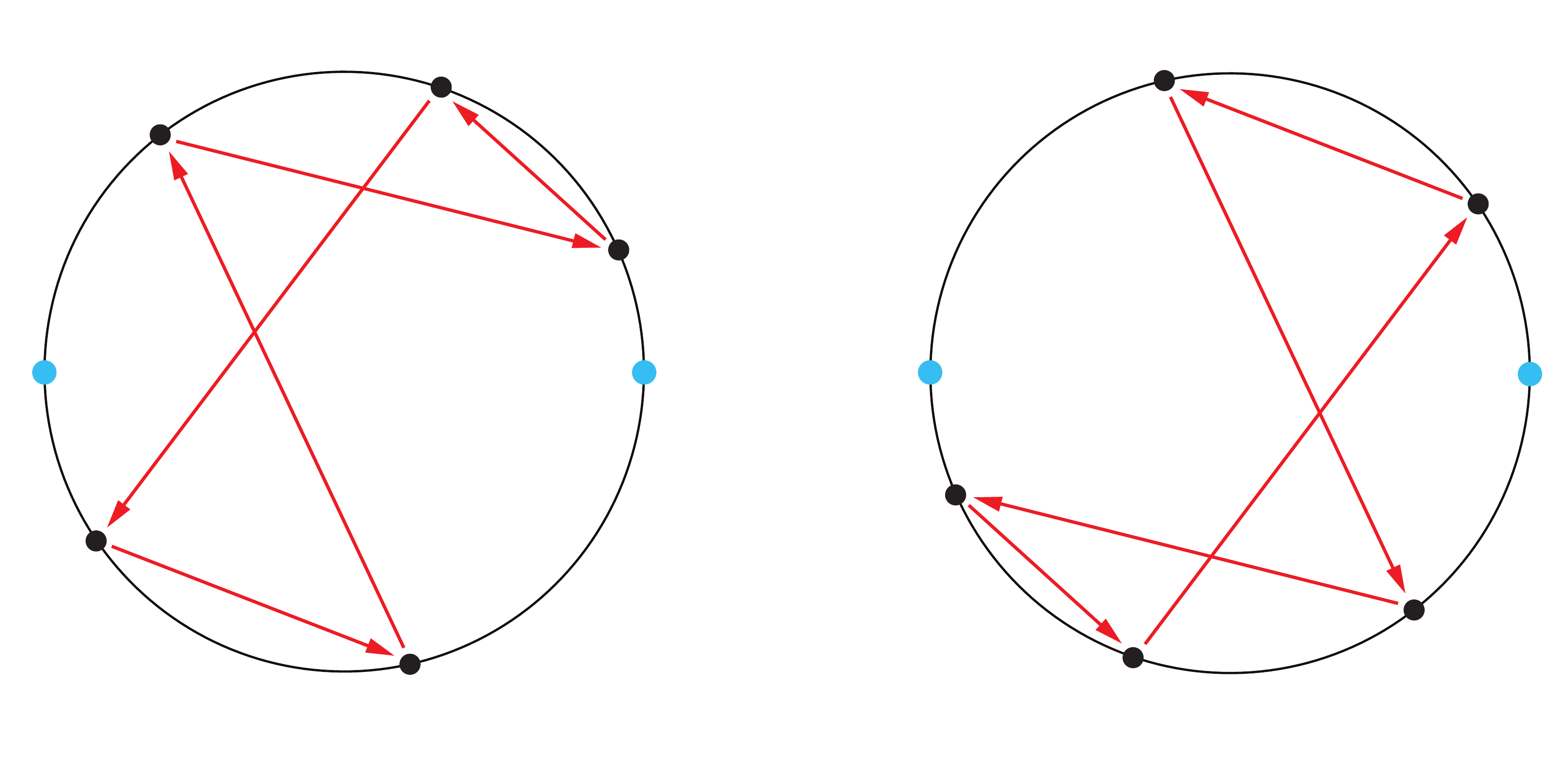}
\put (40.5,33) {\small $16$}
\put (27,45) {\small $48$}
\put (6,41.5) {\small $86$}
\put (0.5,11.5) {\small $144$}
\put (24,2.7) {\small $190$}
\put (95.5,37) {\small $23$}
\put (71,45) {\small $69$}
\put (55,15) {\small $137$}
\put (69,3) {\small $169$}
\put (89,6) {\small $207$}
\end{overpic}
\caption{\sl The two realizations of the combinatorial type $\tau$ of \exref{sinu} under $\mathbf{m}_3$ with $\des(\tau)=3$ and $\sym(\tau)=1$ (angles are shown in multiples of $1/242$). Left: The realization of the representative cycle $\sigma = (1 \ 2 \ 4 \ 5 \ 3)$. Right: The realization of the conjugate cycle $\nu= (1 \ 2 \ 5 \ 3 \ 4)$. The blue dots are the fixed points $0,1/2$ of $\mathbf{m}_3$.}  
\label{re1}
\end{figure}
    
\begin{example}\label{sinv}
Next consider the combinatorial type $\tau$ in $\sC_6$ represented by the cycle $\sigma=(1 \ 2 \ 5 \ 6 \ 3 \ 4)$ with $\sig(\sigma)=(0,1,0,1,0,1)$. Here $\des(\tau)=4$ and $\sym(\tau)=3$. By \thmref{main2}, there is only one realization of $\tau$ under $\mathbf{m}_4$ corresponding to $\sigma= \rho^{-2} \sigma \rho^2 = \rho^{-4} \sigma \rho^4$. The transition matrix $A$ of $\sigma$ and the probability vector solution of $A \bl = 4 \bl$ are 
$$
A = \begin{bmatrix}
0 & 1 & 1 & 1 & 0 & 0 \\
1 & 1 & 1 & 0 & 1 & 1 \\
0 & 0 & 0 & 1 & 1 & 1 \\
1 & 1 & 1 & 1 & 1 & 0 \\
1 & 1 & 0 & 0 & 0 & 1 \\
1 & 0 & 1 & 1 & 1 & 1 \\
\end{bmatrix}
$$
and
$$
\bl = \frac{1}{4^6-1} \, \big( 546, 819, 546, 819, 546, 819 \big).
$$
Using the relations \eqref{deforb}, we obtain the realization 
$$
\OO= \frac{1}{45} \ \big\{ 2, 8, 17, 23, 32, 38 \big\}.
$$
Observe that $\OO=\OO-1/3=\OO-2/3$ (compare \figref{re2}).
\end{example}

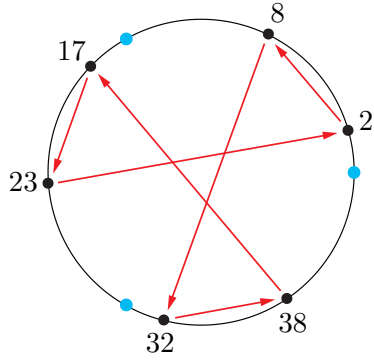
\begin{figure}[t]
\begin{overpic}[width=0.38\textwidth]{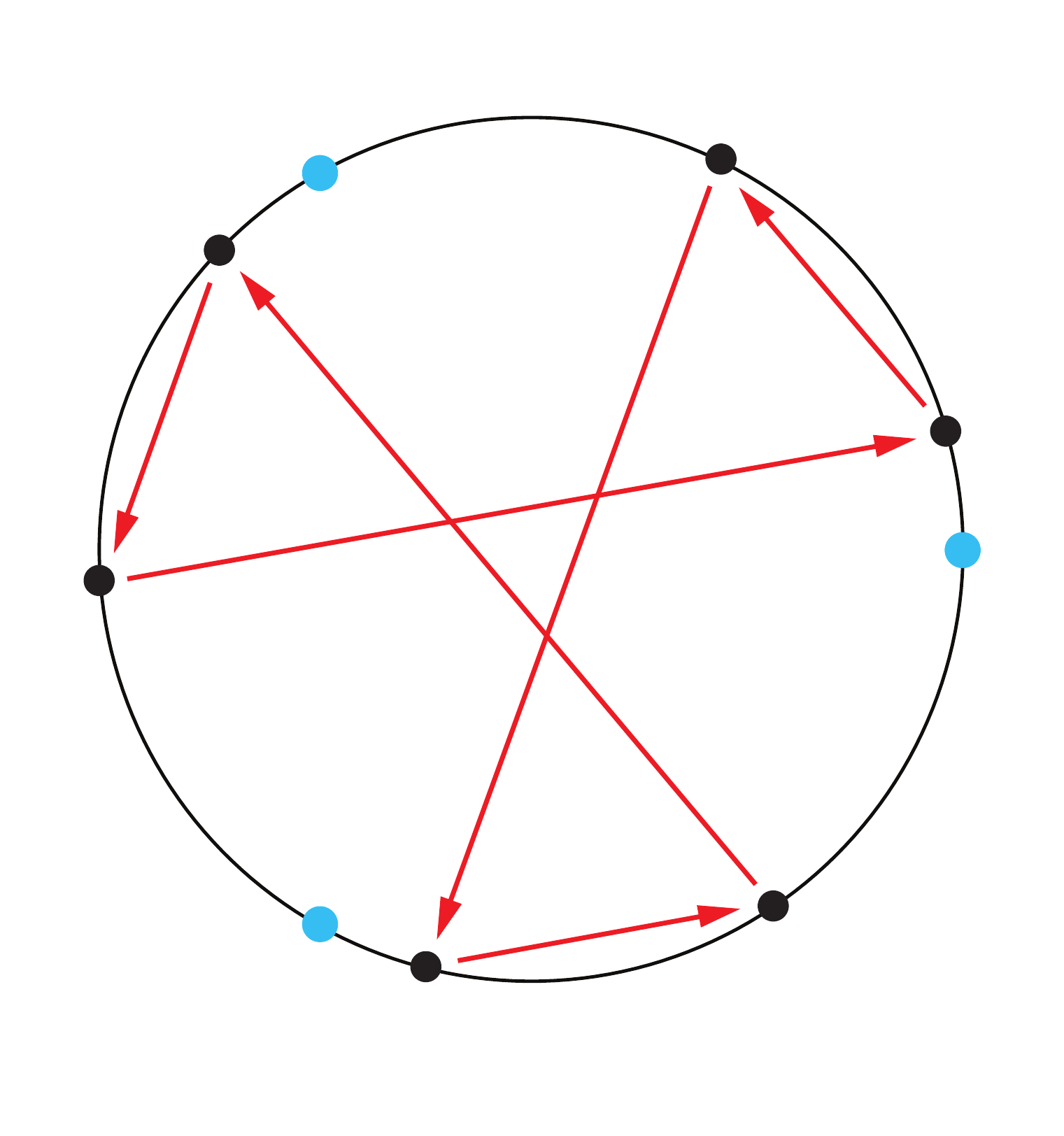}
\put (86,62) {\small $2$}
\put (64,89) {\small $8$}
\put (11,79.5) {\small $17$}
\put (-1,47) {\small $23$}
\put (33,7) {\small $32$}
\put (66,12) {\small $38$}
\end{overpic}
\caption{\sl The unique realization of the combinatorial type $\tau$ of \exref{sinv} under $\mathbf{m}_4$ with $\des(\tau)=4$ and $\sym(\tau)=3$ (angles are shown in multiples of $1/45$). Here the representative cycle is $\sigma = (1 \ 2 \ 5 \ 6 \ 3 \ 4)$. The blue dots are the fixed points $0, 1/3, 2/3$ of $\mathbf{m}_4$.}  
\label{re2}
\end{figure}

\begin{example}\label{sinw}
Finally, consider the combinatorial type $\tau$ in $\sC_8$ represented by the cycle $\sigma=(1 \ 2 \ 4 \ 7 \ 5 \ 6 \ 8 \ 3)$ with $\sig(\sigma)=(0,0,1,1,0,0,1,1)$. Here $\des(\tau)=5$ and $\sym(\tau)=2$. By \thmref{main2}, there are two realizations of $\tau$ under $\mathbf{m}_5$, corresponding to the representative cycles $\sigma = \rho^{-4} \sigma \rho^4$ and $\nu = \rho^{-3} \sigma \rho^3 = \rho^{-7} \sigma \rho^7 = (1 \ 4 \ 2 \ 3 \ 5 \ 8 \ 6 \ 7)$. The transition matrix $A$ of $\sigma$ and the probability vector solution of $A \bl = 5 \bl$ are 
$$
A = \begin{bmatrix}
0 & 1 & 1 & 0 & 0 & 0 & 0 & 0 \\
0 & 0 & 0 & 1 & 1 & 1 & 1 & 1 \\
1 & 1 & 1 & 1 & 1 & 1 & 0 & 0 \\
1 & 1 & 1 & 1 & 1 & 0 & 1 & 1 \\
0 & 0 & 0 & 0 & 0 & 1 & 1 & 0\\
1 & 1 & 1 & 1 & 0 & 0 & 0 & 1 \\
1 & 1 & 0 & 0 & 1 & 1 & 1 & 1 \\
1 & 0 & 1 & 1 & 1 & 1 & 1 & 1 
\end{bmatrix}
$$
and
$$
\bl = \frac{1}{5^8-1} \, \big( 21284, 52584, 53836, 67608, 21284, 52584, 53836, 67608 \big). 
$$ 
\begin{figure}[t]
\begin{overpic}[width=0.8\textwidth]{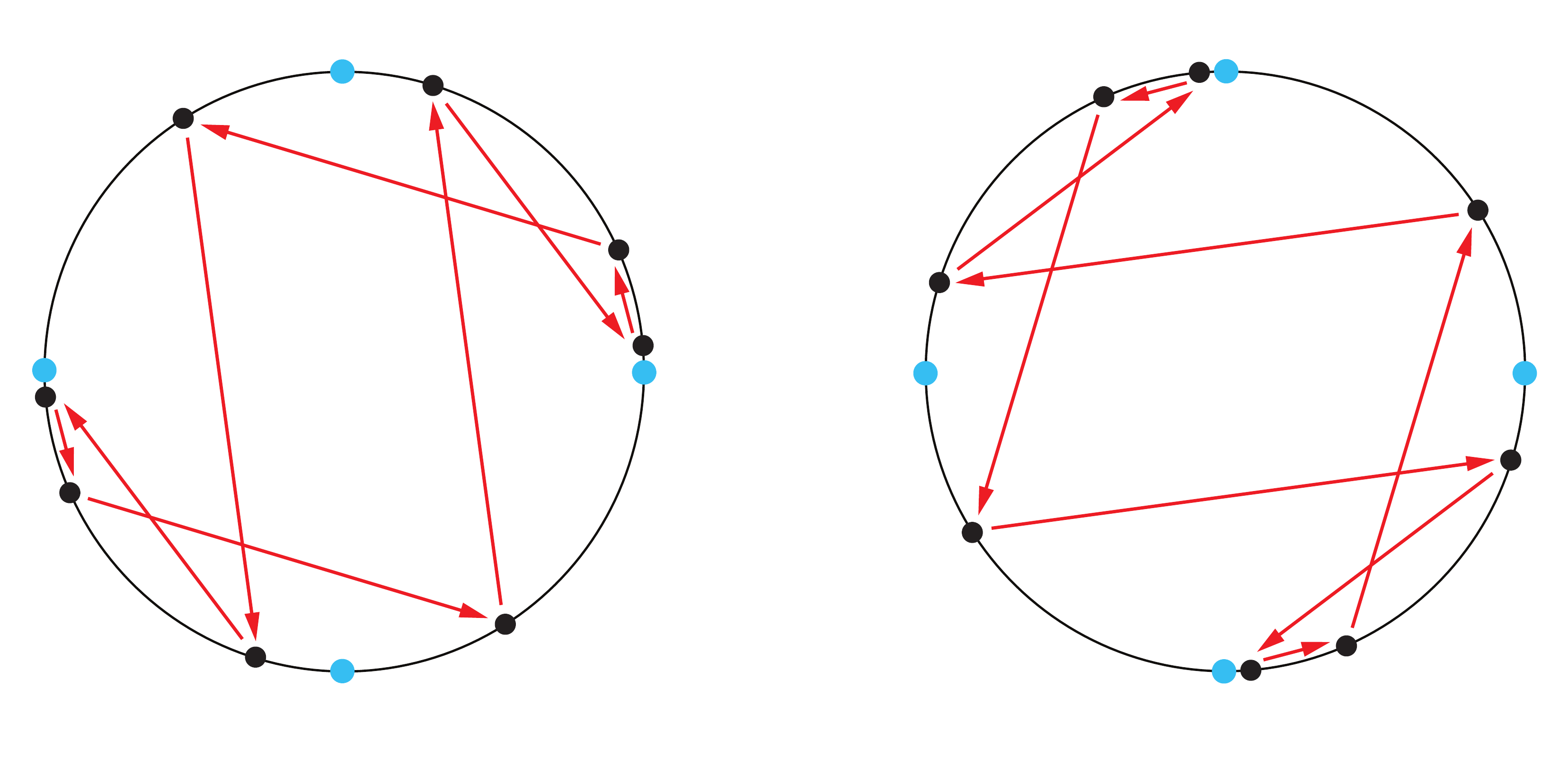}
\put (42.3,26.5) {\small $17$}
\put (41,32.5) {\small $85$}
\put (26,45) {\small $253$}
\put (5.5,42) {\small $425$}
\put (-3,22.5) {\small $641$}
\put (-1.5,15) {\small $709$}
\put (13,3) {\small $877$}
\put (32,5) {\small $1049$}
\put (95.5,35) {\small $113$}
\put (73,46) {\small $329$}
\put (65,43.5) {\small $397$}
\put (54,30.5) {\small $565$}
\put (56.5,12) {\small $737$}
\put (77.3,2.3) {\small $953$}
\put (85.5,4) {\small $1021$}
\put (97.5,19) {\small $1189$}
\end{overpic}
\caption{\sl The two realizations of the combinatorial type $\tau$ of \exref{sinw} under $\mathbf{m}_5$ with $\des(\tau)=5$ and $\sym(\tau)=2$ (angles are shown in multiples of $1/1248$). Left: The realization of the representative cycle $\sigma = (1 \ 2 \ 4 \ 7 \ 5 \ 6 \ 8 \ 3)$. Right: The realization of the conjugate cycle $\nu = (1 \ 4 \ 2 \ 3 \ 5 \ 8 \ 6 \ 7)$. The blue dots are the fixed points $0, 1/4, 2/4, 3/4$ of $\mathbf{m}_5$.}  
\label{re3}
\end{figure}
Using \eqref{deforb}, we obtain the realization 
$$
\OO= \frac{1}{1248} \ \big\{ 17, 85, 253, 425, 641, 709, 877, 1049 \big\}. 
$$
Rotating $\OO$ under $x \mapsto x-1/4 \ (\operatorname{mod} \ \ZZ)$ then gives the realization $\OO'$ of $\nu$:
$$
\OO' = \OO - \frac{1}{4} =  \frac{1}{1248} \ \big\{ 113, 329, 397, 565, 737, 953, 1021, 1189 \big\}. 
$$
Notice that $\OO=\OO-1/2$ and $\OO' = \OO'-1/2$ (compare \figref{re3}).
\end{example}

We end this section with a slightly different view of \thmref{main1} which will be exploited in the next section.  

\begin{definition}
Let $\OO = \{ x_1, \ldots, x_q \}$ be a period $q$ orbit of $\mk$. The {\bit fixed point distribution} of $\OO$ is the integer vector $\fix(\OO)=(n_1,\ldots,n_q)$, where $n_i$ is the number of fixed points of $\mk$ in (the interior of) the partition interval $I_i=[x_i,x_{i+1}]$.\footnote{The fixed point distribution is a dual to the ``deployment vector'' introduced by Goldberg in \cite{G} and further developed in \cite{Z1}; see the discussion at the end of the next section leading to \thmref{main3alt}.}     
\end{definition}

Thus $\fix(\OO)$ records how the $k-1$ fixed points of $\mk$ are distributed among the $q$ points of $\OO$. \vs 

The following properties of $\fix(\OO)=(n_1,\ldots,n_q)$ are immediate: \vs 
\begin{enumerate}
\item[$\bullet$]
$\sum_{i=1}^q n_i=k-1$,  \vs
\item[$\bullet$]
$n_q \geq 1$. \vs
\end{enumerate}

For example, suppose $(\md,\OO)$ is a minimal realization of $\sigma \in \sC_q$ with $\des(\sigma)=d \geq 2$ and $\sig(\sigma)=(a_1, \ldots, a_q)$. Then the interval $I_i$ contains a (unique) fixed point if and only if it is marked. Hence the component $n_i$ of $\fix(\OO)$ is $1$ if $a_i=1$, and is $0$ otherwise. In other words, 
\begin{equation}\label{dep=sig}
\fix(\OO)=\sig(\sigma) \qquad \text{in the minimal case}
\end{equation}
and in particular, $a_q=n_q=1$. \thmref{main1} can be interpreted as saying that this necessary condition is also sufficient, and the resulting realization orbit is unique. \vs

As illustrations of the equation \eqref{dep=sig}, we invite the reader to look back at the five realizations in Examples \ref{sinu}, \ref{sinv}, and \ref{sinw} shown in Figs. \ref{re1}, \ref{re2}, and \ref{re3}, respectively.    

\section{Realizations under $\mk$: The general case}\label{GEN}

We now address the problem of finding realizations of $\sigma \in \sC_q$ under $\mk$ where $k> d=\des(\sigma) \geq 1$ (notice that we are now allowing $\sigma$ to be a rotation cycle). Suppose $(\mk,\OO)$ is such a realization, and consider the minimal realization $(f,\OO)$ defined by mapping each partition interval $I_i$ affinely onto the union $\bigcup_{j \in [\sigma(i),\sigma(i+1))} I_j$. For each $1 \leq i \leq q$ there is an integer $p_i \geq 0$ such that $\mk$ is a $p_i$-winding of $f$ on $I_i$, hence $\sum p_i = k-d$. Recall from the proof of \lemref{des} that $f$ has $d-1$ fixed points, one in each marked interval. It follows from \lemref{fp+k} that the components of $\fix(\OO)=(n_1,\ldots,n_q)$ are given by 
$$
n_i = \begin{cases} p_i+1 & \quad \text{if} \ I_i \ \text{is marked} \\
p_i & \quad \text{if} \ I_i \ \text{is unmarked.}
\end{cases}
$$
In other words, 
\begin{equation}\label{kal}
\fix(\OO)= \sig(\sigma) + \bp, \quad \text{where} \quad \bp=(p_1,\ldots, p_q).  
\end{equation}
In particular $\fix(\OO) \geq \sig(\sigma)$ componentwise. This shows that $\mk$ has at least one fixed point in each of the $d-1$ marked intervals, and at least one in $I_q$ if it is not marked. \vs

It turns out that this necessary restriction on $\fix(\OO)$ is also sufficient. In fact, \thmref{main3} below will show that as long as we put one fixed point of $\mk$ in each marked interval (and one in $I_q$, if it is not marked), then we are free to place the remaining fixed points of $\mk$ among the partition intervals as we wish, and the resulting distribution will be $\fix(\OO)$ for some realization $(\mk,\OO)$ of $\sigma$. To formulate this precisely, we first need the following

\begin{definition}\label{adv}
Let $\sigma \in \sC_q$ with $\des(\sigma)=d$ and let $k>d$. A non-negative integer vector $\bn=(n_1, \ldots, n_q)$ is {\bit $\sigma$-admissible} in degree $k$ if \vs
\begin{enumerate}
\item[$\bullet$]
$\sum_{i=1}^q n_i = k-1$, \vs
\item[$\bullet$]
$n_q \geq 1$, \vs
\item[$\bullet$]
$\bn \geq \sig(\sigma)$ componentwise. \vs
\end{enumerate}     
(The third condition holds automatically if $d=1$.) \vs
\end{definition}

The following is an analog of \thmref{main1}: 

\begin{theorem}\label{main3} 
Let $\sigma \in \sC_q$ with $\des(\sigma)=d$ and let $k>d$. Then, for every $\sigma$-admissible vector $\bn=(n_1,\ldots,n_q)$ in degree $k$, there are $n_q$ realizations $(\mk,\OO)$ of $\sigma$ with $\fix(\OO)=\bn$. More precisely, for each $0 \leq n <n_q$ there is a unique realization  $\OO_n=\{ x_1, \ldots, x_q \}$ with $n$ fixed points of $\mk$ in $(0,x_1)$. As a result, these realizations are rotated copies of one another:
$$
\OO_n =\OO_0+\frac{n}{k-1} \qquad 0 \leq n < n_q.                 
$$   
\end{theorem}

\begin{proof}
The proof is similar to \thmref{main1}. Assume that $\OO=\{ x_1, \ldots, x_q \}$ is a realization of $\sigma$ under $\mk$ with $\fix(\OO)=\bn$, let $I_i=[x_i,x_{i+1}]$ and as before consider the lengths $\ell_i=|I_i|$ and form the probability vector $\bl=(\ell_1,\ldots, \ell_q) \in \RR^q$. Set $\bp= \bn - \sig(\sigma)=(p_1, \ldots, p_q)$. Then \eqref{kal} and \lemref{fp+k} show that $\mk$ is a $p_i$-winding of the piecewise affine minimal realization on $I_i$. Hence, 
\begin{equation}\label{foo1}
p_i + \sum_{j \in [\sigma(i),\sigma(i+1))} \ell_j = k \ell_i \qquad \text{for all} \ i. 
\end{equation}
These $q$ relations can be written as  
\begin{equation}\label{AL1}
B \bl = k \bl,
\end{equation}
where $B$ is the transition matrix of the pair $(\sigma, \bp)$ as defined at the end of \S \ref{TMC} (recall that $B$ is obtained by adding the vector $(p_i,\ldots,p_i) \in \RR^q$ to the $i$-th row of the transition matrix of $\sigma$). By \thmref{PF1}, the equation \eqref{AL1} has a unique solution $\bl$ with positive components. \vs

To construct the orbit $\OO_n=\{ x_1, \ldots, x_q \}$ for $0 \leq n < n_q$, take this unique solution $\bl = (\ell_1, \ldots, \ell_q)$ and define  
\begin{equation}\label{deforb1}
\begin{cases}
x_1 = & \dfrac{n}{k-1}+\dfrac{1}{k-1}  \ds{\sum_{j \in [1,\sigma(1))} \ell_j} \vs \\
x_i = & x_1 \ + \ \ds{\sum_{j \in [1,i)} \ell_j} \qquad \text{for} \ 2 \leq i \leq q.
\end{cases}
\end{equation}
Evidently $0< x_1 < x_2 < \cdots <x_q$. To verify that $0,x_1,\ldots,x_q$ are in positive cyclic order, we need to check $x_q<1$. As in the proof of \thmref{main1}, this follows once we show that $x_1<\ell_q$. We consider two cases: \vs

{\it Case 1.} $a_q=1$ so $p_q=n_q-1$. In this case, $[1,\sigma(1))$ is a proper subset of $[\sigma(q),\sigma(1))$ that does not contain $q$, so by \eqref{deforb1} and \eqref{foo1}, 
$$
(k-1) x_1 = n + \sum_{j \in [1,\sigma(1))} \ell_j < p_q + \ \big( \hspace{-6mm} \sum_{j \in [\sigma(q),\sigma(1))} \ell_j \big) - \ell_q = (k-1) \ell_q. 
$$

{\it Case 2.} $a_q=0$ so $p_q=n_q$. In this case, since 
$$
\sum_{j \in [1,\sigma(1))} \ell_j \leq \sum_{j \in [1,q)} \ell_j = 1-\ell_q,  
$$
we obtain
\begin{align*}
(k-1) x_1 & = n + \sum_{j \in [1,\sigma(1))} \ell_j \leq p_q -1 + \sum_{j \in [1,\sigma(1))} \ell_j \\
& < p_q + \ \big( \hspace{-6mm} \sum_{j \in [\sigma(q),\sigma(1))} \ell_j \big) - \ell_q = (k-1) \ell_q. 
\end{align*}
This shows $0,x_1,\ldots,x_q$ are in positive cyclic order. Since $n/(k-1)<x_1<(n+1)/(k-1)$ by \eqref{deforb1}, there are precisely $n$ fixed points of $\mk$ in $(0,x_1)$. \vs

To see that $(\mk,\OO_n)$ is a realization of $\sigma$, note that by \eqref{deforb1}  
$$
k x_1 = x_1 + n + \sum_{j \in [1,\sigma(1))} \ell_j = x_{\sigma(1)} \qquad (\operatorname{mod} \ \ZZ),
$$
and for $2 \leq i \leq q$,
\begin{align*}
k x_i & = k x_1 + \sum_{j \in [1,i)} k \ell_j \\
& =  x_1 + n + \! \! \sum_{j \in [1,\sigma(1))} \ell_j + \sum_{j \in [1,i)} \Big(p_j + \! \! \sum_{\alpha \in [\sigma(j),\sigma(j+1))} \ell_\alpha \Big) \qquad \text{by} \ \eqref{foo1} \\
& = x_1 + \! \! \sum_{j \in [1,\sigma(i))} \ell_j \qquad (\operatorname{mod} \ \ZZ) \\
& = x_{\sigma(i)}.
\end{align*}
Thus, $\mk(x_i) = x_{\sigma(i)}$ for every $i$. It follows from the relations \eqref{foo1} that $\fix(\OO_n)=\bn$. \vs

Finally, suppose $\OO' = \{ x'_1, \ldots, x'_q \}$ is any realization of $\sigma$ under $\mk$ with $\fix(\OO')=\bn$ and with $n$ fixed points in $(0,x'_1)$. Then the lengths $\ell'_i=x'_{i+1}-x'_i$ must be identical to the above $\ell_i$, as they are uniquely determined by \eqref{foo1}. Since  
$$
(k-1) x'_1 = x'_{\sigma(1)} -x'_1 = \sum_{j \in [1,\sigma(1))} \ell'_j = \sum_{j \in [1,\sigma(1))} \ell_j \qquad (\operatorname{mod} \ \ZZ), 
$$
$x'_1$ and $x_1$ (given by \eqref{deforb1}) must differ by an integer multiple of $1/(k-1)$. Since both $(0,x'_1)$ and $(0,x_1)$ contain $n$ fixed points of $\mk$, we conclude that $x'_1=x_1$.   
\end{proof}

\begin{theorem}\label{main4}
Let $\sigma \in \sC_q$ with $\des(\sigma)=d$ and $\sig(\sigma)=(a_1,\ldots,a_q)$ and let $k>d$. Then the number of realizations of $\sigma$ under $\mk$ is  
\begin{align*} 
& {q+k-d \choose q} & \hspace{-3cm} \text{if} \quad a_q=1, \vs \\
& {q+k-d-1 \choose q} & \hspace{-3cm} \text{if} \quad a_q=0. 
\end{align*}
\end{theorem}

\begin{proof}
By \thmref{main3}, the number of realizations of $\sigma$ under $\mk$ is equal to the number of $\sigma$-admissible vectors $\bn$ in degree $k$, where $\bn=(n_1,\ldots,n_q)$ is counted with multiplicity $n_q$. This is the same as the number of ways to distribute the ``free'' fixed points in the $q+1$ intervals between $0,x_1,\ldots,x_q$. The number of ``free'' fixed points is $k-d$ if $a_q=1$ and $k-d-1$ if $a_q=0$. This leads to the counts $q+k-d \choose q$ and $q+k-d-1 \choose q$, respectively.    
\end{proof}

As a special case, we recover the following result of \cite{G} (see \cite{Z1} for an alternative proof):

\begin{corollary}\label{cased=1}
Every rotation cycle in $\sC_q$ has $q+k-2 \choose q$ realizations under $\mk$. 
\end{corollary}

\begin{proof}
A rotation cycle $\sigma$ has $\des(\sigma)=1$ and $\sig(\sigma)=(0,\ldots,0)$. 
\end{proof}

We can now pass from cycles to combinatorial types. The following is an analog of \thmref{main2}: 

\begin{theorem}\label{main5}
Every combinatorial type $\tau$ in $\sC_q$ with $\des(\tau)=d$ and $\sym(\tau)=s$ has 
$$
\frac{k-1}{s} \ {q+k-d-1 \choose q-1}
$$ 
realizations under $\mk$ if $k>d$. 
\end{theorem}

\begin{proof}
The proof is similar to \thmref{main2}. Pick a representative cycle $\sigma$ for $\tau$ and let $\sig(\sigma)=(a_1,\ldots,a_q)$. Recall that $\tau = \{ \rho^{-1} \sigma \rho, \ldots, \rho^{-r} \sigma \rho^{r}=\sigma \}$, where $r=q/s$, and 
$$
\sig(\rho^{-j} \sigma \rho^j) = (a_{j+1},\ldots,a_q,a_1,\ldots,a_j).
$$
By \thmref{main4}, the number of realizations of the conjugate $\rho^{-j} \sigma \rho^j$ under $\mk$ is ${q+k-d \choose q}$ if $a_j=1$ and is ${q+k-d-1 \choose q}$ if $a_j=0$. By $r$-periodicity \eqref{r-per}, the list $a_1, \ldots, a_r$ contains $(d-1)/s$ entries of $1$ and $(q-d+1)/s$ entries of $0$. Thus, the number of distinct realizations of $\tau$ under $\mk$ is
$$
\frac{d-1}{s} \ {q+k-d \choose q} + \frac{q-d+1}{s} \ {q+k-d-1 \choose q}.
$$
A brief computation reduces this count to $\frac{k-1}{s} \ {q+k-d-1 \choose q-1}$, as claimed.    
\end{proof}

The following three examples illustrate the main results of this section:

\begin{example}\label{sinx}
Consider the realizations under $\mathbf{m}_3$ of the $8$ combinatorial types in $\sC_5$ represented by the cycles in \exref{cyc5}. By \corref{cased=1} (or \thmref{main5} with $s=5$), the rotation cycles $\rho, \rho^2,\rho^3,\rho^4$ each have ${6 \choose 5}=6$ realizations. The combinatorial types represented by $\nu, \nu^{-1}$ have $\des=2, \sym=1$, so by \thmref{main5} they each have $2 \times {5 \choose 4}=10$ realizations. The combinatorial types represented by $\sigma, \sigma^{-1}$ have $\des=3, \sym=1$, so by \thmref{main2} they each have $2$ realizations. These add up to a total of $24+20+4=48$ realizations, which account for all period $5$ orbits under $\mathbf{m}_3$.     
\end{example}

\begin{example}\label{siny}
Consider the combinatorial type $\tau$ of \exref{sinu} represented by the cycle $\sigma=(1 \ 2 \ 4 \ 5 \ 3)$ with $\sig(\sigma)=(0,0,1,0,1)$, $\des(\sigma)=3$, and $\sym(\sigma)=1$. According to \thmref{main5} there are $3 \times {5 \choose 4}=15$ realizations of $\tau$ under $\mathbf{m}_4$. By \thmref{main4}, these consist of ${6 \choose 5}=6$ realizations of $\sigma$, ${6 \choose 5}=6$ realizations of $\rho^{-3}\sigma \rho^3$, and ${5 \choose 5}=1$ realization of each of the conjugates $\rho^{-1}\sigma \rho, \rho^{-2}\sigma \rho^2, \rho^{-4}\sigma \rho^4$. The $6$ realizations of $\sigma$ correspond to the following $\sigma$-admissible vectors in degree $4$: \vs

\begin{align*}
& \bn=(1,0,1,0,1) & & \OO_1= \frac{1}{1023} \ \big\{ 110, 440, 539, 737, 902 \big\} \\[2pt]
& \bn=(0,1,1,0,1) & & \OO_2= \frac{1}{1023} \ \big\{ 46, 184, 523, 736, 898 \big\} \\[2pt]
& \bn=(0,0,2,0,1) & & \OO_3= \frac{1}{1023} \ \big\{ 45, 180, 267, 720, 834 \big\} \\[2pt]
& \bn=(0,0,1,1,1) & & \OO_4= \frac{1}{1023} \ \big\{ 29, 116, 263, 464, 833 \big\} \\[2pt]
& \bn=(0,0,1,0,2) & & \OO_5= \frac{1}{1023} \ \big\{ 25, 100, 262, 400, 577 \big\} \\[2pt]
& & & \OO_6= \frac{1}{1023} \ \big\{ 366, 441, 603, 741, 918 \big\} 
\end{align*}

\noindent
Each orbit $\OO_i$ for $1 \leq i \leq 5$ is obtained by adding the vector $(1,1,1,1,1)$ to the $i$-th row of the transition matrix $A$ of $\sigma$ to form the matrix $B$, solving $B \bl = 4 \bl$ for the probability vector $\ell \in \RR^5$, and using the relations \eqref{deforb1} with $n=0$ to find the $x_i$. The orbit $\OO_6$ is obtained similar to $\OO_5$ but with the choice $n=1$ in \eqref{deforb1}, so $\OO_6=\OO_5+1/3$ (compare \figref{re4}). \vs

The rotations of these $6$ realizations under $x \mapsto x-1/3 \ (\operatorname{mod} \ \ZZ)$ will generate all $15$ realizations of $\tau$, as shown in the table below: \vs \vs

\bgroup
\def\arraystretch{1.2}
\begin{center} 
\begin{tabular}{|c|c|c|c|c|c|}
\hline
cycle & \multicolumn{5}{c|}{realizations} \\
\hline 
\multirow{2}{*}{$\sigma$} & \multirow{2}{*}{$\OO_1$} & \multirow{2}{*}{$\OO_2$} & \multirow{2}{*}{$\OO_3$} & \multirow{2}{*}{$\OO_4$} & $\OO_5$ \\
& & & & & $\OO_5-2/3$ \\
\hline
$\rho^{-1} \sigma \rho$ & $\OO_1-1/3$ & & & & \\
\hline
$\rho^{-2} \sigma \rho^2$ & & $\OO_2-1/3$ & & & \\
\hline
\multirow{2}{*}{$\rho^{-3} \sigma \rho^3$} & \multirow{2}{*}{$\OO_1-2/3$} & \multirow{2}{*}{$\OO_2-2/3$} & $\OO_3-1/3$ & \multirow{2}{*}{$\OO_4-1/3$} & \multirow{2}{*}{$\OO_5-1/3$} \\
& & & $\OO_3-2/3$ & & \\
\hline
$\rho^{-4} \sigma \rho^4$ & & & & $\OO_4-2/3$ & \\
\hline
\end{tabular}
\end{center}
\egroup
\end{example}

\vspace*{3mm}

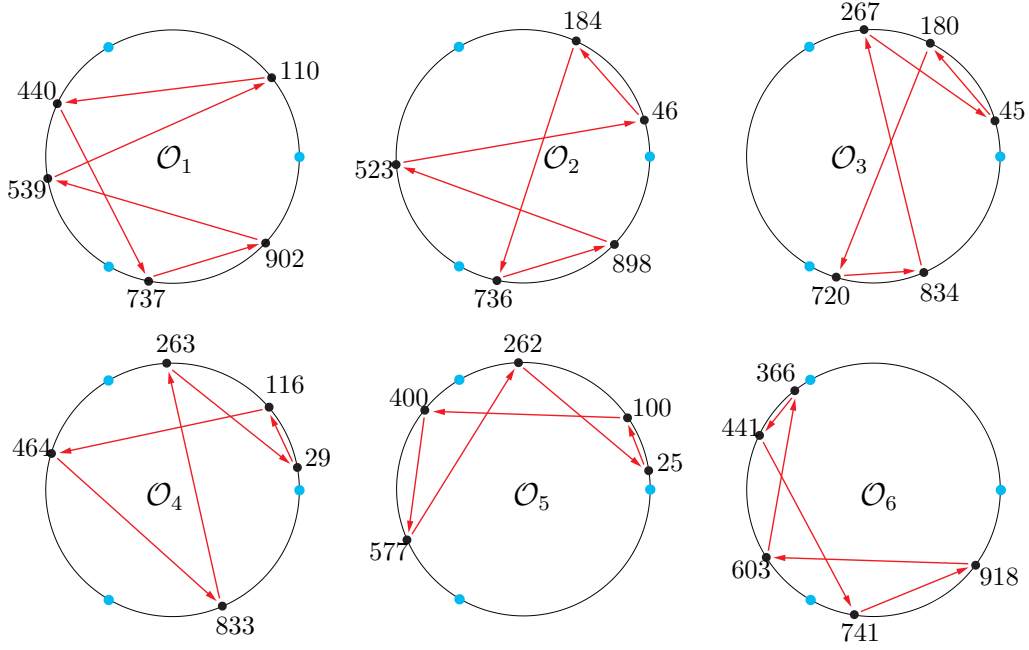
\begin{figure}[t]
\begin{overpic}[width=\textwidth]{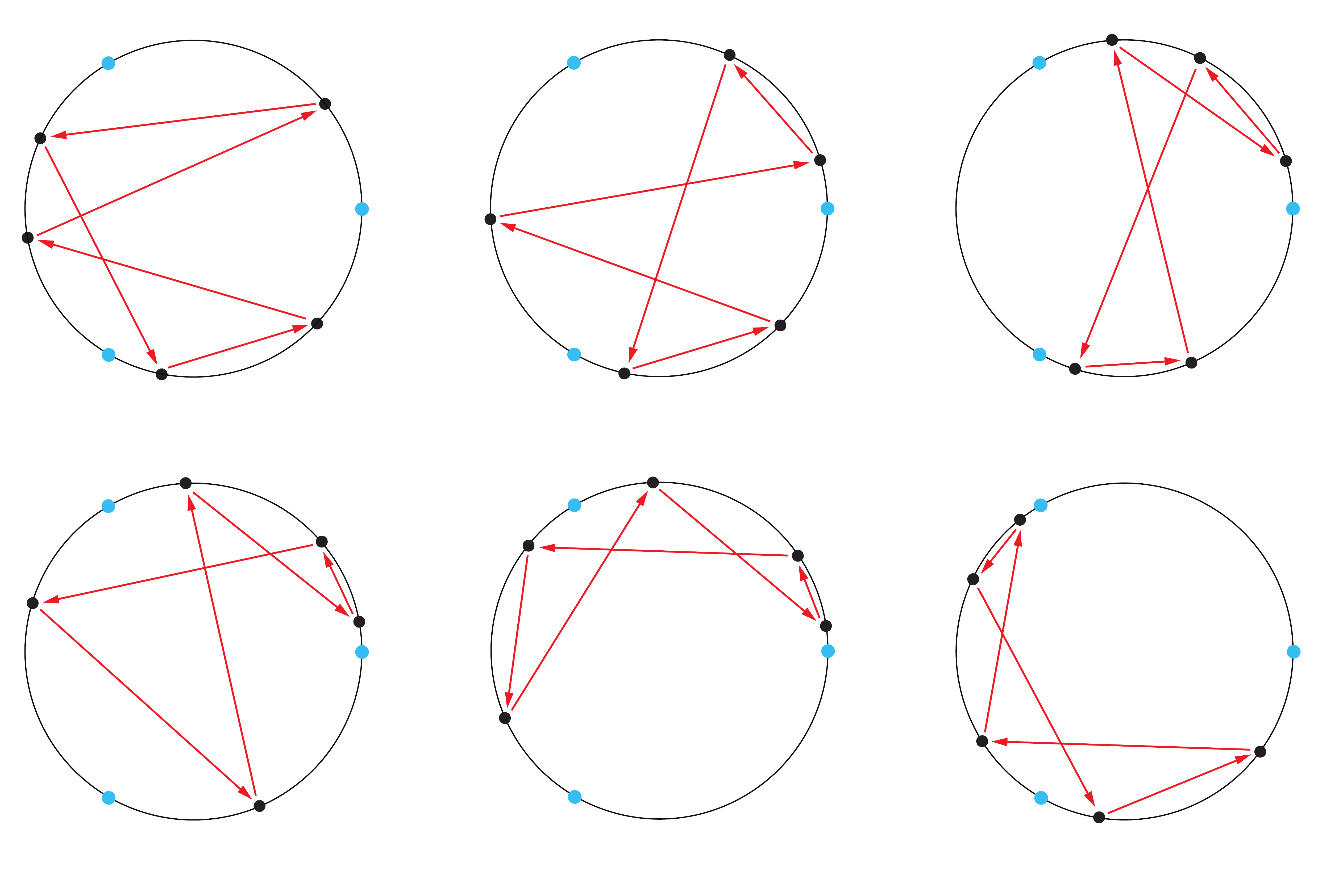}
\put (25.5,60) {\footnotesize $110$}
\put (-1,58) {\footnotesize $440$}
\put (-2,48) {\footnotesize $539$}
\put (10,37) {\footnotesize $737$}
\put (24,41) {\footnotesize $902$}
\put (63,56) {\footnotesize $46$}
\put (54,65) {\footnotesize $184$}
\put (33,50) {\footnotesize $523$}
\put (45,37) {\footnotesize $736$}
\put (59,40.5) {\footnotesize $898$}
\put (98,56) {\footnotesize $45$}
\put (90,64.5) {\footnotesize $180$}
\put (82,66) {\footnotesize $267$}
\put (79,37.3) {\footnotesize $720$}
\put (90,37.6) {\footnotesize $834$}
\put (28,21) {\footnotesize $29$}
\put (24,28.3) {\footnotesize $116$}
\put (13,32.5) {\footnotesize $263$}
\put (-1.5,22) {\footnotesize $464$}
\put (19,4) {\footnotesize $833$}
\put (63.5,20.5) {\footnotesize $25$}
\put (61,26) {\footnotesize $100$}
\put (48,32.5) {\footnotesize $262$}
\put (36.5,27) {\footnotesize $400$}
\put (34.5,11.5) {\footnotesize $577$}
\put (73.5,29.5) {\footnotesize $366$}
\put (70,24) {\footnotesize $441$}
\put (71,9.7) {\footnotesize $603$}
\put (82,3.3) {\footnotesize $741$}
\put (96,9) {\footnotesize $918$}
\put (13,51) {$\OO_1$}
\put (52,51) {$\OO_2$}
\put (81,51) {$\OO_3$}
\put (12,17) {$\OO_4$}
\put (49,17) {$\OO_5$}
\put (84,17) {$\OO_6$}
\end{overpic}
\caption{\sl The six realizations of the cycle $\sigma = (1 \ 2 \ 4 \ 5 \ 3)$ under $\mathbf{m}_4$ with $\des(\sigma)=3$ and $\sym(\sigma)=1$ (angles are shown in multiples of $1/1023$ and the blue dots are the fixed points $0,1/3,2/3$ of $\mathbf{m}_4$). These orbits and their rotated copies under $x \mapsto x-1/3 \ (\operatorname{mod} \ \ZZ)$ produce the $15$ realizations of the combinatorial type $[\sigma]$; see \exref{siny}.}  
\label{re4}
\end{figure}

\begin{example}\label{sinz}
Finally, consider the combinatorial type $\tau$ of \exref{sinv} represented by the cycle $\sigma=(1 \ 2 \ 5 \ 6 \ 3 \ 4)$ with $\sig(\sigma)=(0,1,0,1,0,1)$, $\des(\sigma)=4$ and $\sym(\sigma)=3$. According to \thmref{main5} there are $(4/3) \times {6 \choose 5}=8$ realizations of $\tau$ under $\mathbf{m}_5$. By \thmref{main4}, these consist of ${7 \choose 6}=7$ realizations of $\sigma=\rho^{-2}\sigma \rho^2=\rho^{-4}\sigma \rho^4$ and only ${6 \choose 6}=1$ realization of $\nu=\rho^{-1}\sigma \rho= \rho^{-3}\sigma \rho^3=\rho^{-5}\sigma \rho^5$. The $7$ realizations of $\sigma$ correspond to the following $\sigma$-admissible vectors in degree $5$: \vs

\begin{align*}
& \bn=(1,1,0,1,0,1) & & \OO_1= \frac{1}{15624} \ \big\{ 1113, 5565, 8169, 9597, 12201, 14133 \big\} \\[2pt]
& \bn=(0,2,0,1,0,1) & & \OO_2= \frac{1}{15624} \ \big\{ 488, 2440, 8144, 9472, 12200, 14128 \big\} \\[2pt]
& \bn=(0,1,1,1,0,1) & & \OO_3= \frac{1}{15624} \ \big\{ 483, 2415, 5019, 9471, 12075, 13503 \big\} \\[2pt]
& \bn=(0,1,0,2,0,1) & & \OO_4= \frac{1}{15624} \ \big\{ 482, 2410, 4394, 6346, 12050, 13378 \big\} \\[2pt]
& \bn=(0,1,0,1,1,1) & & \OO_5= \frac{1}{15624} \ \big\{ 357, 1785, 4389, 6321, 8925, 13377 \big\} \\[2pt]
& \bn=(0,1,0,1,0,2) & & \OO_6= \frac{1}{15624} \ \big\{ 332, 1660, 4388, 6316, 8300, 10252 \big\} \\[2pt]
& & & \OO_7= \frac{1}{15624} \ \big\{ 4238, 5566, 8294, 10222, 12206, 14158 \big\} 
\end{align*}

\noindent
Each orbit $\OO_i$ for $1 \leq i \leq 6$ is obtained by adding the vector $(1,1,1,1,1,1)$ to the $i$-th row of the transition matrix $A$ of $\sigma$ to form the matrix $B$, solving $B \bl = 5 \bl$ for the probability vector $\ell \in \RR^6$, and using the relations \eqref{deforb1} with $n=0$ to find the $x_i$. The orbit $\OO_7$ is obtained similar to $\OO_6$ but with the choice $n=1$ in \eqref{deforb1}, so $\OO_7=\OO_6+1/4$. \vs

The rotations of these $7$ realizations under $x \mapsto x-1/4 \ (\operatorname{mod} \ \ZZ)$ will generate the $8$ realizations of $\tau$. In fact, it is easy to see that under this rotation,
\begin{center}
\begin{tikzpicture}[node distance=13mm]
\node (1) {$\OO_1$};
\node (2) [right of =1] {$\OO_8$};
\node (3) [right of =2] {$\OO_5$};
\node (4) [right of =3] {$\OO_3$};
\node (5) [right of =4] {and};
\node (6) [right of =5] {$\OO_2$};
\node (7) [right of =6] {$\OO_6$};
\node (8) [right of =7] {$\OO_7$};
\node (9) [right of =8] {$\OO_4$};
\draw[->,thick,black] (1) to node {} (2);
\draw[->,thick,black] (2) to node {} (3);
\draw[->,thick,black] (3) to node {} (4);
\path[->,thick,black] (4) edge [bend left=] node[above] {} (1);
\draw[->,thick,black] (6) to node {} (7);
\draw[->,thick,black] (7) to node {} (8);
\draw[->,thick,black] (8) to node {} (9);
\path[->,thick,black] (9) edge [bend left] node[above] {} (6);
\end{tikzpicture}
\end{center}
Here $\OO_8=\OO_1-1/4$ is the unique realization of the conjugate cycle $\nu$.
\end{example}

The realizations of a cycle under $\mk$ can be parametrized by another invariant which is a dual to our fixed point distribution. By definition, the {\bit (cumulative) deployment vector} of a period $q$ orbit $\OO$ of $\mk$ is the integer vector 
$$
\dep(\OO)=(w_1,\ldots,w_{k-1}), \quad \text{where} \quad 
w_i = \# \Big( \OO \cap \Big( 0, \dfrac{i}{k-1} \Big) \Big). 
$$ 
Thus, $\dep(\OO)$ tells us how the points of $\OO$ are deployed in between the $k-1$ fixed points of $\mk$ (see \cite{G} and \cite{Z1}). Evidently, 
$$
0 \leq w_1 \leq \cdots \leq w_{k-1}=q.
$$
The vectors $\dep(\OO)$ and $\fix(\OO)$ are related as follows. Suppose $\fix(\OO)=(n_1,\ldots,n_q)$ and $0 \leq n_0 < n_q$ is the number of fixed points of $\mk$ between $0$ and the first point $x_1$ of $\OO$. Denote by $0 \leq j_1<j_2<\cdots<j_e=q$ the indices $j$ for which $n_j>0$. Then,   
\begin{equation}\label{dep-con}
\dep(\OO)=( \, \underbrace{j_1,\ldots,j_1}_{n_{j_1} \ \operatorname{times}}, \, \underbrace{j_2,\ldots,j_2}_{n_{j_2} \ \operatorname{times}}, \, \ldots, \, \underbrace{j_e,\ldots,j_e}_{n_{j_e}-n_0 \ \operatorname{times}} \! \!).
\end{equation}
If $(\mk,\OO)$ is a realization of $\sigma \in \sC_q$ with $\des(\sigma)=d$ 
and $\sig(\sigma)=(a_1,\ldots,a_q)$, then the indices $i_1<\cdots<i_{d-1}$ for which $a_i=1$ appear among the above indices $j_1, \ldots, j_e$. In particular, if $(\md,\OO)$ is the minimal realization of $\sigma$, then $n_0=0$, $n_i=a_i$ for all $1 \leq i \leq q$, $e=d-1$, and $i_\alpha=j_\alpha$ for all $1 \leq \alpha \leq d-1$. Thus, 
$$
\dep(\OO)=(i_1,\ldots, i_{d-1}) \qquad \text{in the minimal case}.
$$
Based on these observations, we see that the deployment vectors $(w_1,\ldots,w_{k-1})$ for realizations of $\sigma$ under $\mk$ satisfy the following two conditions: \vs
\begin{enumerate}
\item[(i)]
$0 \leq w_1 \leq \cdots \leq w_{k-1}=q$, \vs
\item[(ii)]
$i_1, \ldots, i_{d-1}$ appear among the components $w_1, \ldots, w_{k-1}$. \vs
\end{enumerate} 
Note that (ii) holds vacuously when $d=1$. Let us call any integer vector $\bw=(w_1, \ldots, w_{k-1})$ which satisfies these conditions $\sigma$-{\bit admissible} (this is the dual of the notion of $\sigma$-admissible in degree $k$ in \defref{adv} for the fixed point distribution). \vs     

We have thus arrived at the following alternative formulation of Theorems \ref{main1} and \ref{main3}, which we called the ``Realization Theorem'' in \S \ref{intro}: 

\begin{theorem}\label{main3alt}
Let $\sigma \in \sC_q$ with $\des(\sigma)=d$ and $k \geq \max \{ d, 2 \}$. Then, for every $\sigma$-admissible vector $\bw \in \ZZ^{k-1}$ there is a unique realization $(\mk,\OO)$ of $\sigma$ with $\dep(\OO)=\bw$.   
\end{theorem} 

\begin{example}
The realizations $\OO_1,\ldots, \OO_6$ of $\sigma=(1 \ 2 \ 4 \ 5 \ 3)$ under $\mathbf{m}_4$ in \exref{siny} were parametrized by their fixed point distribution, but we can alternatively parametrize them by their deployment vector. Since $\sig(\sigma)=(0,0,1,0,1)$, the allowed deployment vectors $\bw=(w_1,w_2,w_3)$ satisfy $0 \leq w_1 \leq w_2 \leq w_3=5$ and contain $i_1=3,i_2=5$ among their components. Here are all the possibilities: \vs

\bgroup
\def\arraystretch{1.4}
\begin{center} 
\begin{tabular}{|c|c|c|}
\hline
$\bw$ & $\bn$ & realization \\
\hline 
$(1,3,5)$ & $(1,0,1,0,1)$ & $\OO_1$ \\
\hline
$(2,3,5)$ & $(0,1,1,0,1)$ & $\OO_2$ \\
\hline
$(3,3,5)$ & $(0,0,2,0,1)$ & $\OO_3$ \\
\hline
$(3,4,5)$ & $(0,0,1,1,1)$ & $\OO_4$ \\
\hline
$(3,5,5)$ & $(0,0,1,0,2)$ & $\OO_5$ \\
\hline
$(0,3,5)$ & $(0,0,1,0,2)$ & $\OO_6$ \\
\hline
\end{tabular}
\end{center}
\egroup  
\end{example}


\begin{thebibliography}{99}

\bibitem{BS} S. Bullett and P. Sentenac, {\it Ordered orbits of the shift, square roots, and the devil's staircase}, Math. Proc. Cambridge Philos. Soc. {\bf 115} (1994) 451-481. 

\bibitem{DH} A. Douady and J. Hubbard, {\it Etude dynamique des polyn\^{o}mes complexes}, Pr\'epublications math\'emathiques d'Orsay 2/4 (1984/1985). 

\bibitem{F} W. Feller, {\it An Introduction to Probability Theory and Applications}, vol. 1,  3rd ed., Wiley, 1968. 

\bibitem{Ga} F. Gantmacher, {\it Applications of the Theory of Matrices}, Dover Publications, 2005. 
 
\bibitem{G} L. Goldberg, {\it Fixed points of polynomial maps I: Rotation subsets of the circles}, Ann. Sci. \'Ecole Norm. Sup. {\bf 25} (1992) 679-685. 

\bibitem{GM} L. Goldberg and J. Milnor, {\it Fixed points of polynomial maps II: Fixed point portraits}, Ann. Sci. \'Ecole Norm. Sup. {\bf 26} (1993) 51-98. 

\bibitem{GKP} R. Graham, D. Knuth and O. Patashnik, {\it Concrete Mathematics},  2nd ed., Addison-Wesley, 1994. 

\bibitem{GS} C. Grinstead and L. Snell, {\it Introduction to Probability}, 2nd edition, American Mathematical Society, 1997. 

\bibitem{Mc} C. McMullen, {\it Dynamics on the unit disk: Short geodesics and simple cycles}, Comment. Math. Helv. {\bf 85} (2010) 723-749. 

\bibitem{M} J. Milnor, {\it Cubic polynomial maps with periodic critical orbit, Part I}, in Collected Papers of John Milnor, volume VII, American Mathematical Society, 2014, 409-476.    

\bibitem{S} R. Stanley, {\it Enumerative Combinatorics}, vol. 1, Cambridge University Press, 2011. 

\bibitem{T} W. Thurston, {\it On the geometry and dynamics of iterated rational maps}, in Complex Dynamics: Families and Friends, edited by D. Schleicher, A K Peters/CRC Press, 2009.

\bibitem{Z1} S. Zakeri, {\it Rotation Sets and Complex Dynamics}, Lecture Notes in Mathematics, vol. 2214, Springer, 2018. xiv+122 pp.

\bibitem{Z2} S. Zakeri, {\it Cyclic permutations: degrees and combinatorial types}, to appear, 2019. 

\end{thebibliography}
\end{document}